\documentclass[12pt,reqno]{amsart}
\usepackage{amsmath,amssymb,amsthm,amscd}
\usepackage{mathrsfs}
\usepackage[old]{old-arrows} 
\usepackage{stmaryrd}

\usepackage[utf8]{inputenc}
\usepackage[english]{babel}

\usepackage{tikz}
\usepackage{tikz-cd}
\usetikzlibrary{babel} 
\usetikzlibrary{cd} 
\tikzset{
  symbol/.style={
    draw=none,
    every to/.append style={
      edge node={node [sloped, allow upside down,
      auto=false]{$#1$}}}
  }
}

\usepackage{graphicx}
\usepackage[alphabetic]{amsrefs}
\usepackage[bookmarks=false,hypertexnames=false,colorlinks, citecolor=blue]{hyperref}
\usepackage{bookmark}
\usepackage{verbatim,color,geometry}
\usepackage{marginnote}
\usepackage[colorinlistoftodos, textsize=small]{todonotes}
\makeatletter 
\@mparswitchfalse%
\makeatother
\reversemarginpar

\usepackage[shortlabels]{enumitem} 

\newtheorem{thm}{Theorem}[section]
\newtheorem{prop}[thm]{Proposition}
\newtheorem{defn-thm}[thm]{Theorem-Definition}
\newtheorem{defn-lem}[thm]{Lemma-Definition}
\newtheorem{cor}[thm]{Corollary}

\newtheorem{lem}[thm]{Lemma}
\newtheorem*{thm*}{Theorem}
\newtheorem*{prop*}{Proposition}
\newtheorem*{cor*}{Corollary}
\newtheorem*{lem*}{Lemma}

\theoremstyle{definition}

\newtheorem{rem}[thm]{Remark}

\newtheorem*{claim*}{Claim}
\newtheorem*{rem*}{Remark}

\newcommand{\su}{\subset}

\newcommand{\cx}{{\breve{x}}}
\newcommand{\cy}{{\breve{y}}}
\newcommand{\cz}{{\breve{z}}}

\newcommand{\ttto}{\dashrightarrow} 
\newcommand{\tto}{\longrightarrow}

\newcommand{\wt}{\widetilde}

\newcommand{\ov}{\overline}

\newcommand{\Z}{\mathbb{Z}}

\newcommand{\PP}{\mathbb{P}}

\newcommand{\AAA}{\mathbb{A}}

\newcommand{\OO}{\mathcal{O}}

\newcommand{\pp}{\mathfrak{p}}
\newcommand{\fq}{\mathfrak{q}}

\newcommand{\al}{\alpha}
\newcommand{\be}{\beta}
\newcommand{\te}{\theta}
\newcommand{\ga}{\gamma}

\newcommand{\de}{\delta}

\newcommand{\sig}{\sigma}

\newcommand{\ka}{\kappa}

\newcommand{\lam}{\lambda}

\newcommand{\eps}{\varepsilon}

\DeclareMathOperator{\Spec}{Spec}
\DeclareMathOperator{\divv}{div}

\begin{document}

\title[Fibrations by non-smooth projective curves]{Fibrations by plane projective rational quartic \\ curves in characteristic two}

\author{Cesar Hilario}
\address{Mathematisches Institut, Heinrich-Heine-Universit\"at, 40204 D\"usseldorf, Germany}
\email{cesar.hilario.pm@gmail.com}

\author{Karl-Otto Stöhr}
\address{IMPA, Estrada Dona Castorina 110, 22460-320 Rio de Janeiro, Brazil}
\email{stohr@impa.br}



\dedicatory{October 10, 2025}


\begin{abstract}
We give a complete classification, up to birational equivalence, of all fibrations by plane projective rational quartic curves in characteristic two.
\end{abstract}

\maketitle

\section{Introduction}

\setcounter{tocdepth}{1}

In this paper we investigate in characteristic $p=2$ the birational geometry of fibrations by rational curves of degree $d=4$ in the projective plane.

The case of degree $d=2$ corresponds to \emph{conic bundles}, which have a long and rich history that goes back to work by the Italian school (e.g., the theory of ruled surfaces), and to more recent work related to the birational geometry of complex threefolds 
(see the expository article \cite{Pro18}).
Over the past few years, conic bundles have also been studied in positive characteristic
(see \cite{JVV24} and the references therein). 

If the fibres are of degree $d>2$, then their arithmetic genus $g = (d-1)(d-2)/2$ is larger than their geometric genus $\ov g = 0$, and so they admit singularities.
By Bertini's theorem this can only happen in characteristic $p>0$.
More precisely, 
as follows from Tate's genus change formula \cite{Tate52},
the prime $p$ must be equal to $m+1$ where $m$ is a divisor of the integer $2 (g - \ov g) = (d-1)(d-2)$.

If $d=3$ then $p\in\{2,3\}$, and we obtain the so-called \emph{quasi-elliptic fibrations}, which are fibrations whose general fibres are plane cubic curves with a cusp.
Equivalently, the generic fibre $C=f^{-1}(\eta) = T_\eta$ of a quasi-elliptic fibration $f:T\to B$ is a \emph{quasi-elliptic curve} over the function field $K=k(B)$ of the base $B$, i.e., $C$ is a regular proper geometrically integral curve of arithmetic genus $g=1$ over $K$.
Quasi-elliptic fibrations play a key role in the extension of the Enriques classification of complex algebraic surfaces to positive characteristics, accomplished by Bombieri and Mumford \cite{BM76,BM77}.

If $d=4$, then $p\in\{2,3,7\}$.
The cases $p=3$ and $p=7$ were investigated by Salomão \cite{Sal11,Sal14} and the second author \cite{St04}.
In the present paper we give a birational classification of the case $p=2$.
In other words, we classify, up to birational equivalence, the fibrations by plane projective rational quartic curves in characteristic $p=2$.
Our main result asserts that the generic fibre $C=T_\eta$ of such a fibration $T \to B$, which is a curve over the function field $K=k(B)$, falls into one of five disjoint classes of curves.




\begin{thm} \label{2024_06_01_23:20}
Let $C$ be a regular proper non-hyperelliptic geometrically rational curve over a field $K$ of characteristic $p=2$.
Assume that $C$ has arithmetic genus $h^1(\OO_C)=3$.
Then $C$ is isomorphic to a plane projective quartic curve over $K$ defined by one of the following equations
\begin{enumerate}[\upshape (i)]
    \item \label{2024_06_01_23:21}
    $y^4 + a z^4 + x z^3 + b x^2 z^2 + c x^4 = 0$, \\
    where $a,b,c \in K$ are constants satisfying $c \notin K^2$;
    
    \item \label{2024_06_01_23:22}
    $ y^4 + a z^4 + b x^2 y^2 + c x^2 z^2 + b x^3 z + d x^4 = 0$, \\
    where $a,b,c,d \in K$ are constants satisfying $a\notin K^2$ and $b\neq 0$;
    
    \item \label{2024_06_01_23:23}
    $b y^4 + d z^4 + y^2 z^2 + x z^3 + (b + b^2 c^3) x^2 z^2 + a x^2 y^2 + a x^3 z + (a b^2 c^3 + a^2 d) x^4 = 0$, \\
    where $a,b,c,d\in K$ are constants satisfying $a\notin K^2$ and $b,c\neq0$;
    
    \item \label{2024_06_01_23:24}
    $y^4 + a z^4 + x z^3 + b x^3 z + c x^4 = 0$, \\
    where $a,b,c \in K$ are constants satisfying $b \notin K^2$;

    \item \label{2024_06_01_23:25}
    $ y^4 + d z^2 y^2 + (c+a) z^4 + d x z^3 + bd\, x^2 y^2 + x^2 z^2 + bd\, x^3 z + b^2 c x^4 = 0 $, \\
    where $a,b,c,d\in K$ are constants satisfying $a,b\notin K^2$ and $d\neq 0$.
\end{enumerate}
Conversely, each of these equations defines a curve of the above type.
\end{thm}

The first two classes of curves were studied in our previous article \cite{HiSt23}, and in this paper we determine the remaining three (see Theorem~\ref{2024_05_27_00:10}).
To complete the classification we decide when two curves in the same class are isomorphic
(see \cite[Propositions~3.5 and~4.5]{HiSt23} and Proposition~\ref{2024_07_22_21:20}), 
and as a by-product obtain that in the cases~\ref{2024_06_01_23:22}, \ref{2024_06_01_23:23} and~\ref{2024_06_01_23:25} the polynomial expressions $a b^2 + c^2$, $b c^3$ and $a b^2 d^2$ are invariants of the curve $C$, respectively.

Furthermore, we show that each family of curves is distinguished by three intrinsic properties, as documented in Table~\ref{2025_04_03_02:45}.
In this table $\pp$ denotes the only non-smooth point on the regular curve $C$ (see Section~\ref{2024_04_05_18:25} for details), which, viewed as a Weil divisor on $C$, can be canonical (\ref{2024_06_01_23:21} and~\ref{2024_06_01_23:22}) or non-canonical (\ref{2024_06_01_23:23}, \ref{2024_06_01_23:24} and~\ref{2024_06_01_23:25}).
For each $n\geq 0$ we denote by $\pp_n$ the image of $\pp$ in the regular curve $C_n|K$, which is defined as the normalization of the $n$-th iterated Frobenius pullback $C^{(p^n)}|K$ of $C|K$.
Note that there is an infinite chain of relative Frobenius morphisms over $K$
\[ C_0 = C \to C_1 \to C_2 \to C_3 \to \cdots. \]
In all five cases the image point $\pp_n$ is non-smooth for $n=1$ and smooth for $n \geq 2$. By the main theorem in \cite{HiSt22}, the smooth point $\pp_n$ is actually rational if $n \geq 3$.
But $\pp_2$ is rational only in cases~\ref{2024_06_01_23:21} and~\ref{2024_06_01_23:23}.

\addtolength{\tabcolsep}{6pt}
\begin{table}[h]
\begin{center}
\begin{tabular}{c c c c}
    \hline
    & the divisor $\pp$ is canonical & the point $\pp_2$ is $K$-rational & $E=K(C_2)$ \\ \hline
    (i) & Yes & Yes & Yes \\
    (ii) & Yes & No & No \\
    (iii) & No & Yes & No \\ 
    (iv) & No & No & Yes \\
    (v) & No & No & No \\
    \hline
\end{tabular}
\vspace{\abovecaptionskip}
\caption{Comparison of intrinsic properties of the five classes of curves.}
\label{2025_04_03_02:45}
\end{center}
\end{table}

The first two columns in Table~\ref{2025_04_03_02:45} do not provide a distinction between the last two classes of curves.
This motivates us to introduce and study a second canonical field of the regular curve $C|K$.
Recall that the \emph{canonical field} of $C|K$ is the subfield of $F=K(C)$ generated over $K$ by the quotients of all non-zero holomorphic differentials of $C$.
Since $C$ is non-hyperelliptic, this field coincides with $F$ in all five cases.
We define the \emph{pseudocanonical field} of $C|K$ as the subfield $E$ of $F$ generated over $K$ by the quotients of all non-zero \emph{exact} holomorphic differentials of $C$.
We show that the field extension $E\su F$ has degree $4=p^2$ in all cases, and that it is purely inseparable, i.e., $E=K(C_2)$, only in cases~\ref{2024_06_01_23:21} and~\ref{2024_06_01_23:24}.

The curves $C|K$ in the theorem exhibit the following interesting properties: the normalized Frobenius pullback $C_n|K$ of $C|K$ is a rational curve for $n \geq 3$, a smooth curve of genus zero for $n=2$, and a quasi-elliptic curve for $n=1$. 
In light of the purely inseparable Frobenius map $C \to C_1$, the latter implies that in characteristic $p=2$ every fibration by plane projective rational quartic curves arises as a degree $p$ inseparable cover of a quasi-elliptic fibration (see also Corollary~\ref{2025_10_10_16:30}).
This is a unique feature of geometry in characteristic $p=2$, for in characteristic $p > 2$ the normalized Frobenius pullback $X_1$ of any regular curve $X$ of genus $h^1(\OO_X) = 3$ is smooth (see \cite[Corollary~2.7]{HiSt22}), and therefore not quasi-elliptic.

To prove our results we work in the arithmetic setting of function field theory. 
The proof of Theorem~\ref{2024_06_01_23:20} goes as follows:
first we determine a presentation of the function field $F|K = K(C)|K$ of the regular curve $C|K$, and secondly we find a realization of $C$ as a plane curve of degree $2g-2=4$ in $\PP^{g-1}(K) = \PP^2(K)$, 
where $g=3$ is the arithmetic genus of $C$, 
through the sections of a canonical divisor.
The determination of a presentation of $F|K$ is based on the Riemann--Roch theorem, the Bedoya--Stöhr algorithm \cite{BedSt87} and the main theorem in \cite{HiSt22}.
If the only non-smooth point $\pp$ on $C$ is a canonical divisor, then its sections provide the presentation of $F|K$ and the realization of $C$ as a plane quartic curve over $K$ (see \cite{HiSt23}).


In this paper we focus on the much harder case where the divisor $\pp$ is not canonical.
Here the presentation of $F|K$ has to be obtained by looking at the Riemann--Roch spaces $H^0(\pp^r)$ of the powers $\pp^r$ of $\pp$, since $\pp$ itself does not have enough sections (see the proofs of Theorems~\ref{2024_02_18_00:35} and~\ref{2021_05_21_18:40}).
In addition, since $\pp$ is not canonical the realization of $C$ as a plane quartic curve requires the prior determination of a canonical divisor, whose sections fullfil such a realization. We perform this task by using differentials (see Section~\ref{2024_04_05_18:30}).

If $C|K$ is the generic fibre of a fibration $T\to B$, then the behaviour of most special fibres is governed by the \emph{geometric generic fibre} $C_{\ov K} = C \otimes_K \ov K$.
This is a rational plane quartic curve over the algebraic closure $\ov K = \ov {k(B)}$, with a unique singular point that is unibranch and lies over the non-smooth point $\pp \in C$. The quartic curve $C_{\ov K}$ is strange, i.e., all its tangent lines meet in a common intersection point, and it has the remarkable property that its tangent lines  are either all bitangents (\ref{2024_06_01_23:22}, \ref{2024_06_01_23:23} and~\ref{2024_06_01_23:25}) or all non-ordinary inflection tangents (\ref{2024_06_01_23:21} and~\ref{2024_06_01_23:24}).

The explicit description in Theorem~\ref{2024_06_01_23:20} allows us to construct five fibrations that are universal in the sense that any fibration $T\to B$ by plane projective rational quartic curves is obtained, up to birational equivalence, from one of them by a base extension (see \cite[Theorems~5.1 and~5.2]{HiSt23} and Theorem~\ref{2024_05_29_21:05}). 
We prove that the total spaces of these fibrations are uniruled, and more generally, that the total space of any fibration by (possibly singular) rational curves is uniruled (see Proposition~\ref{2024_09_26_19:35}).

If the base of the fibration is one-dimensional then we obtain a smooth surface $S$ together with a proper surjective morphism $S\to B$ to a curve $B$, such that almost every fibre is a plane rational quartic curve with a unique singular point.
This is reminiscent of the theory of elliptic surfaces \cite{SchSh10}, where almost all fibres are elliptic curves, or of the theory of quasi-elliptic surfaces \cite{BM76,Lan79}, where almost all fibres are plane cuspidal cubic curves.
The fibres of $S\to B$ that are integral are finite in number, and we call them the \emph{bad fibres} of the fibration.
\footnote{Following Kodaira's classification of singular fibres on elliptic surfaces, one may be tempted to call them \emph{singular fibres}, but this may be misleading because in this paper each fibre has singularities.}
Mirroring the theory of elliptic surfaces, it is natural to restrict to surfaces $S$ that are \emph{relatively minimal} over $B$, in the sense that no bad fibre contains smooth rational curves of self-intersection $-1$.

Since the genus $h^1(C)=3$ of the generic fibre $C=S_\eta$ is positive,
a theorem of Lichtenbaum and Shafarevich (see \cite[Theorem~4.4]{Lic68}, \cite[p.\,155]{Sha66}, or \cite[p.\,422]{Liu02}) guarantees that the fibration $S\to B$, and therefore the fibres, are uniquely determined by the generic fibre $C|K$.
However, this provides no information on the types of bad fibres that may arise.
For elliptic and quasi-elliptic surfaces the bad fibres were classified by Kodaira \cite{Kod63} and Néron \cite{Ner64} (see also \cite[Chapter~4]{CDL23}).
For a fibration $S\to B$ by plane rational quartic curves we determined in \cite[Section~3]{HiSt22} and \cite[Section~6]{HiSt23} the bad fibres in specific situations, namely for two pencils coming from items~\ref{2024_06_01_23:22} and \ref{2024_06_01_23:21} in Theorem~\ref{2024_06_01_23:20} respectively.
In Section~\ref{2024_06_02_00:30} of the present paper we analyze the bad fibres of a fibration coming from item~\ref{2024_06_01_23:24}, whose configurations are slightly more involved.
The general picture, however, remains largely unexplored, and we expect that our explicit description of all possible generic fibres will shed further light on this question.

\section{Geometrically rational function fields of genus three in characteristic two}
\label{2024_04_05_18:25}

Given a regular proper geometrically integral curve $C$ over a field $K$ of characteristic $p\geq0$, its function field $F|K = K(C)|K$ is a \emph{one-dimensional separable function field}, that is, $F|K$ is a separably generated field extension of transcendence degree $1$, with $K$ algebraically closed in $F$.
Conversely, every one-dimensional separable function field $F|K$ is the function field of some curve $C|K$ of the above type.

Let $F|K$ be a one-dimensional separable function field of genus $g=3$.
We assume that $F|K$ is geometrically rational, that is, the extended one-dimensional separable function field $\ov K F | \ov K = \ov K \otimes_K F | \ov K$ has genus $\ov g = 0$. 
The strict inequality $\ov g < g$ can only occur in characteristic $p>0$, in which case the genus drop $g - \ov g$ is a multiple of $(p-1)/2$ by Tate's genus change formula \cite{Tate52}, so we conclude $p\in \{ 2,3,7 \}$.
The cases $p=3$ and $p=7$ were studied by Salomão \cite{Sal11,Sal14} and the second author \cite{St04}.
In this section we assume that $p=2$.

For each $n \geq 0$ let $g_n$ be the genus of the $n$-th \emph{Frobenius pullback} $F_n|K := F^{p^n}{\cdot}K|K$ of $F|K$, where $F_n = F^{p^n}{\cdot}K$ is the only intermediate field of $F|K$ such that $F|F_n$ is purely inseparable of degree $p^n$.
Note that $F_n|K$ is the function field of the $n$-th normalized Frobenius pullback $C_n|K = \wt{C^{(p^n)}}|K$ of $C|K$.


According to \cite[Corollary~2.7~(iii)]{HiSt22} we have $g_1 \leq 1$ and $g_n = \ov g = 0$ for $n\geq 2$.
If $g_1=0$, then $F_1|K$ will be a quadratic subfield of genus zero of $F|K$, hence $F|K$ will be hyperelliptic.
Therefore, as in this paper we are interested in non-hyperelliptic function fields, we assume throughout that the Frobenius pullback $F_1|K$ has genus $g_1=1$, i.e., $F_1|K$ is a quasi-elliptic function field (see \cite[Section~2]{HiSt23}).

In view of \cite[Proposition~2.4]{HiSt22} and Rosenlicht's genus drop formula \cite[Formula~2.3]{HiSt22}, the assumption $g_1=1$ means that there exists a unique singular prime $\pp$ in $F|K$,
whose restrictions $\pp_n$ to the Frobenius pullbacks $F_n|K$ have geometric singularity degrees $\de(\pp) = 3$, $\de(\pp_1)=1$ and $\de(\pp_n)=0$ for $n\geq2$.
In particular $\pp_1$ is the only singular prime of the quasi-elliptic Frobenius pullback $F_1|K$.

Moreover, by \cite[Corollary~2.19]{HiSt22} the singular prime $\pp$ is non-decomposed, i.e., there is a unique prime in $\ov K F|\ov K$ that lies over $\pp$, and so by \cite[Theorem~2.24]{HiSt22} the restricted prime $\pp_n$ is rational for $n\geq3$.
In particular, for each $n \geq 3$ the genus zero function field $F_n|K$ is rational.

In this section we study the function fields $F|K$ that satisfy the above properties plus the additional condition that the divisor $\pp$ is not canonical.
The case where the divisor $\pp$ is canonical was analyzed in \cite[Section~3]{HiSt23}.
We divide the discussion into two major parts, treating first the case where the non-singular restricted prime $\pp_2$ is rational.

\begin{thm}\label{2024_02_18_00:35}
A one-dimensional separable function field $F|K$ of characteristic $p=2$ and genus $g = 3$ is geometrically rational and admits a prime $\pp$ such that $\de(\pp) = 3$, $\de(\pp_1) = 1$, $\pp_2$ is rational, and such that $\pp$ is not a canonical divisor, if and only if 
$F=K(x,z,y)$ is generated by three functions $x$, $z$, and $y$ that can be put into the following normal form
\begin{align*}
z^2 &= (c_0 + c_1 x + x^2) (c_0 A_2 + c_1^{-1} + c_1 A_2 x + A_2 x^2), \\
y^2 &= (c_0 + c_1 x + x^2) (B_0 + B_1 x + z),
\end{align*}
where $c_0,c_1,A_2,B_0,B_1\in K$ are constants satisfying the conditions $c_1\neq 0$ and $A_2\notin K^2$.
The singular prime $\pp$ is the only pole of the function $x$.
It has 
degrees $\deg(\pp)=4$, $\deg(\pp_1)=2$, $\deg(\pp_2)=1$, and
residue fields $\ka(\pp)=\ka(A_2^{1/4})$, $\ka(\pp_1)=K(A_2^{1/2})$, $\ka(\pp_2)=K$.
\end{thm}

The theorem complements \cite[Theorem~3.1~(i)]{HiSt23}, which characterizes the function fields $F|K$ such that $\pp_2$ is rational and $\pp$ is a canonical divisor.
The proof will rely on \cite[Section~2]{HiSt23}, which provides a normal form for the quasi-elliptic Frobenius pullback $F_1|K$.
Another key ingredient will be the Bedoya--Stöhr algorithm \cite{BedSt87}, which enables us to compute several local invariants of the primes of $F|K$.

Note that in order to apply this algorithm for a given prime $\fq$, all we need is that its restriction $\fq_n$ to $F_n|K$ be rational for some $n$, i.e., that $\fq$ be non-decomposed, a condition that is automatic if we assume, as it is assumed in \cite{BedSt87}, that the base field $K$ is separably closed (see \cite[Corollary~2.17]{HiSt22}).

\begin{proof}
Let $F|K$ be a function field of genus $g = 3$ and let $\pp$ be a prime such that $\de(\pp) = 3$, $\de(\pp_1) = 1$, $\pp_2$ is rational, and such that the divisor $\pp$ is non-canonical. 
(Note that by the genus drop formula \cite[Formula~2.3]{HiSt22} the existence of the singular prime $\pp$ ensures that $F|K$ is geometrically rational.)
As the Frobenius pullback $F_1|K$ is quasi-elliptic and the restriction $\pp_2$ of its only singular prime $\pp_1$ is rational, 
we deduce from \cite[Theorem~2.1~(i)]{HiSt23} that $F_1|K$ admits the following normal form
\[ F_1|K = K(x,z)|K, \, \text{ where } \, z^2 = a_0 + x + a_2 x^2 + a_4 x^4, \text{ and $a_0,a_2 \in K$, $a_4 \in K \setminus K^2$}. \]
The singular prime $\pp_1$ is the only pole of the function $x$, and it has residue fields
\[ \ka(\pp_1)=K(a_4^{1/2}), \quad \ka(\pp_2)=K . \]
In particular 
$\deg(\pp_1) = p \, \deg(\pp_2) = 2$,
or in other words,
the prime $\pp_1$ is inertial (or unramified) over $F_2=K(x)$.
The function $z$ is a separating variable 
of $F_1|K$, that is, the finite extension $F_1|K(z)$ is separable, i.e., $F_1=F_2(z)$, i.e., $z\notin F_2$.
These functions satisfy the incidence properties $x\in H^0(\pp_2) \setminus K$ and $z \in H^0(\pp_1^2) \setminus H^0(\pp_2^2)$, or more precisely
\begin{align*}
    H^0(\pp_1) &= H^0(\pp_2) = K \oplus Kx, \\
    H^0(\pp_2^2) &=  K \oplus Kx \oplus Kx^2, \\
    H^0(\pp_1^2) &= K \oplus Kx \oplus Kx^2 \oplus Kz;
\end{align*}
see \cite[Remark 2.3]{HiSt23}. Furthermore, from $\dim H^0(\pp_1^4)=8$ we obtain
\begin{equation}\label{2023_12_10_18:45}
H^0(\pp_1^4) = K \oplus Kx \oplus Kx^2 \oplus Kx^3 \oplus Kx^4 \oplus Kz \oplus Kxz \oplus Kx^2 z. 
\end{equation}

Let 
$e$ denote the ramification index of the extension $\pp|\pp_1$.
As
the divisor $\pp^e$ has degree 
\[ \deg(\pp^e) = [F:F_1] \cdot \deg(\pp_1) = 4 = 2g-2,  \] 
it follows from the Riemann--Roch theorem that the spaces of global sections $H^0(\pp^{ne})$ of the divisors $\pp^{ne}$ have dimension
\begin{align*}
    \dim H^0(\pp^{ne}) &= 4n-2 \quad \text{if $n \geq 2$},\\
    \dim H^0(\pp^e) &= 
    \begin{cases} 3 & \text{if $\pp^e$ is a canonical divisor},\\
    2 & \text{if $\pp^e$ is not a canonical divisor}. \end{cases}
\end{align*}
In particular, as the space $H^0(\pp^e)$ contains the 2-dimensional vector space $H^0(\pp_1)$, the divisor $\pp^e$ is non-canonical if and only if $H^0(\pp^e)=H^0(\pp_1)= K \oplus Kx$.

We claim that the divisor $\pp^e$ is non-canonical, i.e., $H^0(\pp^e)=K\oplus Kx$.
Indeed, assume for contradiction that there is a function $y\in F$ such that
\[ H^0(\pp^e) = H^0(\pp_1) \oplus Ky = K \oplus Kx \oplus Ky. \]
This function does not belong to $F_1=K(x,z)$ because $H^0(\pp^e) \cap F_1 = H^0(\pp_1)$,
hence $F=F_1(y)=K(x,z,y)$, or in other words, $y$ is a separating variable of $F|K$.
Since the square $y^2$ lies in $H^0(\pp^{2e}) \cap F_1 = H^0(\pp_1^2)$, but not in $F_2=K(x)$ as $y^2$ is a separating variable of $F_1|K$,
there are constants $b_i \in K$ with $b_3\neq 0$ such that $y^2 = b_0 + b_1 x + b_2 x^2 + b_3 z$.
As the residue class $\frac yx(\pp) \in \ka(\pp)$ lies outside $\ka(\pp_1)=K(a_4^{1/2})$, because $\frac yx(\pp)^2 = b_2 + b_3 a_4^{1/2} \notin K$, we conclude $\ka(\pp)=K(a_4^{1/2}, \frac yx (\pp))  \supsetneqq \ka(\pp_1)$, whence $e=1$ and the divisor $\pp=\pp^e$ would be canonical, in contradiction to the assumptions. This proves the claim.

We want to find a presentation of $F|K$.
Since $H^0(\pp^e)=H^0(\pp_1)$, to get a generator of the extension $F|F_1$ we must pass to $\pp^{2e}$ and $\pp_1^2$.
As $\dim H^0(\pp^{2e})=6 > \dim H^0(\pp_1^2)=4$ there is an element $y\in H^0(\pp^{2e})\setminus H^0(\pp_1^2)$, which does not belong to $F_1=K(x,z)$ because $H^0(\pp^{2e}) \cap F_1 = H^0(\pp_1^2)$.
Therefore $y$ is a separating variable of $F|K$, that is, \[ F=F_1(y) = K(x,z,y). \]
Since the square $y^2$ lies in $H^0(\pp^{4e})\cap F_1 = H^0(\pp_1^4)$, but not in $F_2=K(x)$ as $y^2$ is a separating variable of $F_1|K$, 
there exist constants $b_i,c_i\in K$ with $(c_0,c_1,c_2)\neq (0,0,0)$ such that 
\[ 
y^2 
= b_0 + b_1 x + b_2 x^2 + b_3 x^3 + b_4 x^4 + (c_0 + c_1 x + c_2 x^2) z. 
\]
In order to study the singular prime $\pp$ 
we introduce the functions $\cx:=x^{-1}\in F_2 = K(x)$, $\cz:=zx^{-2} \in F_1=K(x,z)$ and $\cy:= yx^{-2} \in F$.
Note that $\cx$ is a local parameter at both $\pp_1$ and $\pp_2$, and that $\cz$ and $\cy$ satisfy the relations
\begin{align*}
    \cz^2 &= a_4 + a_2 \cx^2 + \cx^3 + a_0 \cx^4, \\
    \cy^2 &= b_4 + b_3 \cx + b_2 \cx^2 + b_1 \cx^3 + b_0 \cx^4 + (c_2 + c_1 \cx + c_0 \cx^2) \cz. 
\end{align*}
In particular, for the residue classes $\cz(\pp),\cy(\pp) \in \ka(\pp)$ we have
\[ \cz(\pp)^2 = a_4 \notin K^2, \quad \cy(\pp)^2 = b_4 + c_2 \cz(\pp),\quad \ka(\pp_1) = K(\cz(\pp)) . \]

We claim that $\cy(\pp)$ does not belong to $\ka(\pp_1)$.
Indeed, assume the contrary $\breve y(\pp) \in \ka(\pp_1)$.
Then $c_2=0$ since $\cz(\pp)\notin K$, 
and $\cy(\pp) = \al + \be \cz(\pp)$ for some $\al,\be \in K$. 
Substituting $y$ with $y + \al x^2 + \be z$ we may assume $\cy(\pp)=0$, i.e., $b_4=0$.
If $b_3 + c_1 \cz(\pp) \neq 0$, then $v_{\pp_1}(\cy^2)=1$ and $\pp$ is ramified over $F_1$ (i.e., $e=2$) with local parameter $\cy$, so that $\de(\pp) = 2 \de(\pp_1) + \frac 12 v_{\pp_2} (d\cy^4) = 2 + \frac 12 v_{\pp_2}((c_1^2 \cx^4 + c_0^2 \cx^6)d\cx) > 3$ by \cite[Theorem~2.3]{BedSt87}, a contradiction.
In the opposite case $b_3 + c_1 \cz(\pp) = 0$ we have $b_3=c_1=0$ (and therefore $c_0\neq 0$) since $\cz(\pp)\notin K$, hence the function $y^2 = b_0 + b_1 x + b_2 x^2 + c_0 z$ belongs to $H^0(\pp_1^2) = H^0(\pp^{2e}) \cap F_1$ and thus $y\in H^0(\pp^e) \setminus H^0(\pp_1)$, which contradicts the fact that the divisor $\pp^e$ is non-canonical.
This proves the claim.

It follows from the claim that $e=1$, or more precisely, the prime $\pp$ is inertial over $F_1$ with residue field $\ka(\pp)=K(\cz(\pp),\cy(\pp))$.
Now, by \cite[Theorem~2.3]{BedSt87} the hypothesis $\de(\pp)=3$ means that $2\de(\pp_1) + \frac 12 v_{\pp_2}(d\cy^4)=3$, i.e.,
the differential $d\cy^4 = (c_2^2 + c_1^2 \cx^2 + c_0^2 \cx^4) \cx^2 d\cx$ of $F_2|K$ has order $2$ at $\pp_2$, i.e., $c_2\neq 0$, in which case we may normalize $c_2=1$ by replacing $x,y,z$ with $c_2^2 x,c_2^3 y,c_2 z$ respectively.

We have thus integrated the assumption $\de(\pp)=3$ into our normal form.
To complete the proof it remains to translate the two conditions that $g=3$ and that the divisor $\pp$ is non-canonical into relations between the coefficients $a_i,b_i,c_i$.
Since $F|K=K(x,y)|K$ with $y^4 = f(x)$, it follows from
the Jacobian criterion \cite[Corollaries~4.5 and~4.6]{Sal11} and the genus drop formula \cite[Formula~2.3]{HiSt22}
that the assumption $g=3$ is satisfied if and only if the zeroes of the function $\frac{dy^4}{dx}=f'(x)=(c_0 + c_1 x + x^2)^2$ are non-singular primes, that is, for every zero $\fq$ of the function
\[ c(x) := c_0 + c_1 x + x^2 \]
we have $\de(\fq)=0$.
Note that our normal form already ensures that the restricted primes $\fq_1$ have $\de(\fq_1)=0$, because $F_1|K=K(x,z)|K$ is quasi-elliptic.

We claim that $c_1\neq0$.
Seeking a contradiction we suppose $c_1=0$.
Assume first that the root $c_0^{1/2}$ of the polynomial $c(x) = c_0 + x^2$ belongs to $K$.
By assumption, the zero $\fq$ of the function $x+c_0^{1/2}$ is non-singular, i.e., $\de(\fq)=0$.
Substituting $x$ with $x+c_0^{1/2}$ we can normalize $c_0=0$, i.e., $x(\fq)=0$, i.e., $x$ is a local parameter at the rational prime $\fq_2$ of $F_2|K=K(x)|K$.
Using \cite[Proposition~4.1]{BedSt87} we deduce that the non-singular prime $\fq_1$ is rational (and ramified over $F_2$) if and only if $z(\fq)=a_0^{1/2}$ lies in $K$.
In particular $\ka(\fq_1)=K(z(\fq))$.
Moreover $y(\fq)\in\ka(\fq_1)$, since otherwise $\fq$ is inertial over $F_1$ and $\de(\fq) = \frac 12 v_{\fq_2}(dy^4) = \frac12 v_{\fq_2}(x^4 dx) = 2 > 0$ by \cite[Theorem~2.3]{BedSt87}, 
and thus by subtracting from $y$ an element of $K+Kz$ we can normalize $y(\fq)=0$, i.e., $b_0=0$.

When $\fq_1$ is not rational we have $b_1=0$, because otherwise $\fq$ is ramified over $F_1$ with local parameter $y$ and $\de(\fq)= \frac12 v_{\fq_2}(dy^4)=2$ by \cite[Theorem~2.3]{BedSt87};
then the function $\big( \frac yx \big)^2 = b_2 + b_3 x + b_4 x^2 + z$ belongs to $H^0(\pp_1^2) = H^0(\pp^2) \cap F_1$ and so $\frac {y}{x} \in H^0(\pp)\setminus H^0(\pp_1)$, a contradiction because the divisor $\pp=\pp^e$ is non-canonical.
When $\fq_1$ is rational, i.e., $z(\fq)=a_0^{1/2} \in K$, we may normalize $a_0=0$ by subtracting $a_0^{1/2}$ from $z$. Then $z$ is a local parameter at $\fq_1$ and $v_{\fq_1}(y^2 + b_1 z^2) \geq 4$.
Since $v_{\fq_1}(dy^2) = v_{\fq_1}(x^2 dz)=4$ as the differential $dx$ of $F_1|K=K(x,z)|K$ vanishes, it follows from \cite[Proposition~4.1]{BedSt87} and $\de(\fq)=0$ that $b_1\in K^2$, hence we can normalize $b_1=0$ by replacing $y$ with $y + b_1^{1/2}z$.
As before, this yields the contradiction $\frac yx \in H^0(\pp)\setminus H^0(\pp_1)$.

Thus in the proof of the claim we can suppose that the root
$c_0^{1/2}$ of the polynomial $c(x)=c_0 + x^2$ does not belong to $K$.
By our hypothesis, the zero $\fq$ of the function $\tau:=c_0 + x^2 \in F$ is non-singular, i.e., $\de(\fq)=0$.
Moreover, it is clear that $\tau$ is a local parameter at the rational prime $\fq_3$ of $F_3|K=K(\tau)|K$, and that $\fq_2$ is unramified over $F_3$ with $\ka(\fq_2)=K(x(\fq)) = K(c_0^{1/2})$.
Since $z(\fq)\notin \ka(\fq_2)$ as 
$z(\fq)^2 \notin K$,
the prime $\fq_1$ is unramified over $F_2$ with $\ka(\fq_1) = K(x(\fq),z(\fq))$.
Now, if $y(\fq)\notin \ka(\fq_1)$ then $\fq$ is inertial over $F_1$ and $\de({\fq})=\frac12 v_{\fq_3}(dy^8)=\frac12 v_{\fq_3}(\tau^4 d\tau) = 2$ by \cite[Theorem~2.3]{BedSt87}, a contradiction.
In the opposite case $y(\fq) \in K(x(\fq),z(\fq))$, say $t(\fq)=0$ for some $t$ in $y + K + K x + K z + K xz$, the prime $\fq$ is ramified over $F_1$ with local parameter $t$ because 
\[ v_{\fq_3}(dt^8) = v_{\fq_3}(dy^8) =4<8, \]
and therefore $\de(\fq)=\frac12 v_{\fq_3}(dt^8)=2$, a contradiction.
This completes the proof of the claim that $c_1 \neq 0$.


We next normalize $b_4=0$ by replacing $z$ with $z+ b_4 x^2$. 
Now we claim that
\[ g=3 \ \text{ if and only if } \ a(r),a(s),b(r),b(s)\in L^2,\]
where $a(x) := a_0 + x + a_2 x^2 + a_4 x^4$, $b(x) := b_0 + b_1 x + b_2 x^2 + b_3 x^3$, and $r,s \in L$ are the two roots of the polynomial $c(x)$ in the separable closure $L$ of $K$.
To see this we may assume that $K$ is separably closed, i.e., $K=L$, by passing from $K$ to $L$ if necessary.
Then the function $c(x)$ has exactly two zeros (one for each root $r,s \in K$), and we must show that for every such zero $\fq$ the following holds
\[ \de(\fq)=0 \ \text{ if and only if } \ z(\fq),y(\fq)\in K. \]
Let $r$ be the root of $c(x)$ that corresponds to the zero $\fq$.
As $x(\fq) = r \in K$, by subtracting $r$ from $x$ we can suppose $x(\fq)=0$,
that is, $c_0 = 0$ and $x$ is a local parameter at the rational prime $\fq_2$ of $F_2|K=K(x)|K$.
Since 
\[ v_{\fq_2}(d y^4) = v_{\fq_2}\big( (c_1^2 x^2 + x^4)dx \big) = 2 > 0, \]
we deduce from \cite[Theorem~2.3]{BedSt87} that $y(\fq)\in \ka(\fq_1)$ whenever $\de(\fq)=0$.
Assuming that $z(\fq)\notin K$, we see that the prime $\fq_1$ is unramified over $F_2$ with $\ka({\fq_1}) = K(z(\fq))$, and if we suppose $\de(\fq)=0$ then $y(\fq)\in \ka({\fq_1})$ means that 
$t(\fq)=0$ for some $t$ in $y + K + Kz$,
so that $\fq$ is ramified over $F_1$ with local parameter $t$ because
\[ v_{\fq_2}(dt^4) = v_{\fq_2}(dy^4) = 2 < 4, \]
and therefore $\de(\fq)=\frac12 v_{\fq_2}(dt^4)=1$, a contradiction.
Thus the condition $\de(\fq)=0$ implies that $z(\fq)\in K$.
So in order to prove the claim we may assume $z(\fq)=0$, i.e., $a_0 = 0$,
in which case $\fq_1$ is ramified (and therefore rational) over $F_2$ with local parameter $z$.
Since $v_{\fq_1}(dy^2)=v_{\fq_1}((c_1 x + x^2)dz)=2$ as the differential $dx$ of $F_1|K = K(x,z)|K$ vanishes, we conclude from \cite[Proposition~4.1]{BedSt87} that $\de(\fq) = 0$ if and only if $y(\fq)\in K$, thereby proving the claim.

We next rewrite the conditions $a(r),a(s),b(r),b(s)\in L^2$ on the roots $r,s \in L$ of the polynomial $c(x)$ in terms of the constants $a_i,b_i,c_i \in K$. 
To this end we apply the theory of symmetric polynomials.
Define
\[ q:= r+s = c_1 \in K, \qquad t := rs = c_0 \in K. \]
Clearly, the four symmetric polynomial expressions
\begin{align*}
    a(r)+a(s) &=q+a_2 q^2 + a_4 q^4, \\
    r^2 a(r) + s^2 a(s) &= a_0 q^2 + (q^3 + qt) + a_2 q^4 + a_4 (q^6 + q^2 t^2), \\
    b(r) + b(s) &= b_1 q + b_2 q^2 + b_3 (q^3 + qt), \\ 
    r^2 b(r) + s^2 b(s) &= b_0 q^2 + b_1 (q^3 + qt) + b_2 q^4 + b_3 (q^5 + t (q^3 + qt)), 
\end{align*}
belong to $L^2\cap K =K^2$, say they can be written as $\al^2$, $\be^2$, $\te^2$, $\ga^2$ respectively. Since $q\neq 0$ we can perform four normalizations along the following steps: 
substitute $z$ with $z+ \frac \al q x$, so that $a(r)+a(s)=0$; 
replace $z$ with $z+\frac\be q$, so that $r^2 a(r) + s^2 a(s) = 0$; 
substitute $y$ with $y + \frac \te q x$, so that $b(r)+b(s)=0$; 
replace $y$ with $y + \frac \ga q$, so that $r^2 b(r) + s^2 b(s) =0$.
Thus
\[ a(r) + a(s) = r^2 a(r) + s^2 a(s) = b(r) + b(s) = r^2 b(r) + s^2 b(s) = 0, \]
i.e., $a(r)=a(s)=b(r)=b(s)=0$, which means that $c(x)$ divides both $a(x)$ and $b(x)$. 
We have therefore obtained a normal form for $F|K$ as in the statement of the theorem.

To complete the proof of the theorem we must verify that the normal form ensures that the divisor $\pp$ is non-canonical, i.e., $H^0(\pp) = K\oplus Kx$.
To do this we first find the space of global sections $H^0(\pp^2)$ of the divisor $\pp^2$.
Since the $6$-dimensional vector space $H^0(\pp^2)$ contains the $4$-dimensional vector space $H^0(\pp_1^2)$ and the function $y$, we must find a sixth element $u\in H^0(\pp^2)$ such that
\begin{equation*}
    H^0(\pp^2)=K\oplus Kx\oplus Kx^2 \oplus Kz \oplus Ky \oplus Ku.
\end{equation*}
Write $c(x):= c_0 + c_1 x + x^2$, $A(x):= c_0 A_2 + c_1^{-1} + c_1 A_2 x + A_2 x^2$ and $B(x) := B_0 + B_1 x$, so that $z^2 = c(x) A(x)$ and $y^2 = c(x) (B(x) + z)$.
We claim that $u:=\frac{yz}{c(x)}$ satisfies the desired property. Indeed, since $u^2=A(x) (B(x)+z)$ lies in $H^0(\pp_1^4)$ (see~\eqref{2023_12_10_18:45}),
and hence in $H^0(\pp^4)$, it is clear that $u\in H^0(\pp^2)$. Moreover, the functions $1,x,x^2,z,y,u$ are linearly independent over $K$ because their squares $1,x^2,x^4,c(x) A(x),c(x)(B(x)+z), A(x)(B(x)+z)$ are linearly independent over $K^2$ (recall that $A_2\notin K^2$).
    
We finally show that $H^0(\pp) = K\oplus Kx$. We must prove that each element $h$ of $H^0(\pp)$ lies in $K\oplus Kx$. 
Since $H^0(\pp)$ is contained in $H^0(\pp^2)$, we may write $h=d_1 + d_2 x + d_3 x^2 + d_4 z + d_5 y + d_6 u$, so that
\[ h^2 = d_1^2 + d_2^2 x^2 + d_3^2 x^4 + d_4^2 c(x) A(x) + d_5^2 c(x) (B(x) + z) + d_6^2 A(x) (B(x) + z)\]
lies in $H^0(\pp^2) \cap F_1 = H^0(\pp_1^2)=K\oplus Kx \oplus Kx^2 \oplus Kz$. 
Using the fact that $A_2\notin K^2$ we conclude $d_3=d_4=d_5=d_6=0$, that is, $h\in K\oplus Kx$, as desired.
\end{proof}

\begin{rem}\label{2024_08_26_14:35}
We draw some consequences from the above proof.
Let $F|K=K(x,z,y)|K$ be a function field as in Theorem~\ref{2024_02_18_00:35}.
Then the first Frobenius pullback $F_1|K=K(x,z)|K$ is a quasi-elliptic function field,
given as in \cite[Theorem~2.1~(i)]{HiSt23}, and the second Frobenius pullback $F_2|K=K(x)|K$ is a rational function field.
Also $e=e_1=1$, where $e$ and $e_1$ denote the ramification indices of $\pp$ and $\pp_1$ over $F$ and $F_1$ respectively. Besides,
\begin{align*}
    H^0(\pp_1) &= K \oplus Kx, \\
    H^0(\pp_2^2) &=  K \oplus Kx \oplus Kx^2, \\
    H^0(\pp_1^2) &= K \oplus Kx \oplus Kx^2 \oplus Kz, \\
    H^0(\pp^2) &= K\oplus Kx\oplus Kx^2 \oplus Kz \oplus Ky \oplus Ku,
\end{align*}
where $u = yz (c_0 + c_1 x + x^2)^{-1}$.
\end{rem}

The above Riemann--Roch spaces allow us to determine the isomorphism classes of the function fields in the theorem.

\begin{prop}\label{2024_03_25_19:55}
Let $F|K$ and $F'|K$ be two function fields as in Theorem~\ref{2024_02_18_00:35}. 
Then $F|K$ and $F'|K$ are isomorphic if and only if there exist constants $\al,\mu_2,\mu_3,\mu_4,\mu_5 \in K$ satisfying $(\mu_4,\mu_5)\neq (0,0)$ and
\begin{align*}
    \eps^{-3} c_0' &= \al^2 \eps + \al \eps c_1 + \mu_5^2 c_1^{-1} + \eps c_0, \qquad c_1' = \eps^2 c_1, \qquad \eps^6 A_2' = A_2 + \ga^2, \\
    B_0' &= (\al B_1 + B_0) \eps + c_1^{-1} (\mu_4 \mu_5 + \mu_3^2), \qquad \eps B_1' = B_1, 
\end{align*}
where $\eps:= \mu_4^2 + \mu_5^2 A_2 \neq 0$ and $\ga := \eps^{-1}(\mu_2^2 + \mu_3^2 A_2)$.
The corresponding $K$-isomorphisms $F'\overset\sim\to F$ are given by the transformations
\[ (x',z',y') \mapsto (\eps^2(\al + x),\eps(\be + \ga c_1 x + \ga x^2 + z),\eps^2(\tau + \mu_2 c_1 x + \mu_2 x^2 + \mu_3 z + \mu_4 y + \mu_5 u)), \]
where $\be:= \ga c_0 + \eps^{-1} c_1^{-1} \mu_5 (\mu_4 + \mu_5 \ga)$ and $\tau := \mu_2 c_0 + \eps^{-1} c_1^{-1} \mu_5 (\mu_2 \mu_5 + \mu_3 \mu_4)$.
\end{prop}

\begin{proof}
Every $K$-isomorphism $\sig:F' \overset\sim\to F$ preserves the only singular primes $\pp'$ and $\pp$ of $F'|K$ and $F|K$ respectively, hence it induces an isomorphism $H^0(\pp_m'^n) \overset\sim\to H^0(\pp_m^n)$ for each $m,n \geq 0$.
Thus the incidence properties of $x'$, $z'$ and $y'$ inherit as follows
\begin{align*}
    \sig(x') \in H^0(\pp_1) \setminus K , \quad
    \sig(z') \in H^0(\pp_1^2) \setminus H^0(\pp_2^2), \quad
    \sig(y') \in H^0(\pp^2) \setminus H^0(\pp_1^2).
\end{align*}
Moreover, the functions $\sig(x')$, $\sig(z')$, $\sig(y')$ also satisfy the two polynomial equations with the coefficients $c_i',B_i',A_2'$.
In these equations we substitute $\sig(x')$, $\sig(z')$, $\sig(y')$ by the corresponding $K$-linear combinations of $1$, $x$, $x^2$, $z$, $y$, $u$, and we replace $z^2$ and $y^2$ with the right-hand sides of the equations in the announcement of Theorem~\ref{2024_02_18_00:35}.
As the eighth functions $1$, $x$, $x^2$, $x^3$, $x^4$, $z$, $xz$, $x^2 z$ are $K$-linearly independent, we obtain $2\cdot 8 = 16$ polynomial equations between $A_2$, $b_i$, $c_i$, $A_2'$, $b_i'$, $c_i'$, and the $2 + 4 + 6 = 12$ coefficients of the expansions of $\sig(x')$, $\sig(z')$, $\sig(y')$.
Some of these equations are identically zero.
\end{proof}

\begin{cor}\label{2024_03_25_11:50}
Let $F|K$ be a function field as in Theorem~\ref{2024_02_18_00:35}.
Then the product $\iota := c_1 B_1^2$ is an invariant of $F|K$.
If $\iota \neq 0$ then the group $\mathrm{Aut}(F|K)$ of automorphisms of $F|K$ is trivial.
If $\iota = 0$ then $\mathrm{Aut}(F|K)$ is isomorphic to $\Z/2\Z$, and it is generated by the transformation
\[ (x,z,y) \mapsto (x + c_1, z, y). \]
\end{cor}


We now turn to the function fields $F|K$ whose only singular primes $\pp$ have the property that their restrictions $\pp_2$ are non-rational.
The theorem below complements \cite[Theorem~3.1~(ii)]{HiSt23}.

\begin{thm}\label{2021_05_21_18:40}
A one-dimensional separable function field $F|K$ of characteristic $p=2$ and genus $g = 3$ is geometrically rational and admits a prime $\pp$ such that $\de(\pp) = 3$, $\de(\pp_1) = 1$, $\pp_2$ is non-rational, and such that $\pp$ is not a canonical divisor, if and only if 
$F|K$ can be put into one of the following normal forms
\begin{enumerate}[\upshape (a)]
    \setcounter{enumi}{1}
    \item \label{2021_05_21_18:42}
    $F|K = K(x,w,z,y)|K$, where
    \[ w^2 = x + a_2 x^2, \quad z^2 = b_0 + b_2 x^2 + w, \quad y^2 = xz, \]
    and $a_2,b_0,b_2 \in K$ are constants satisfying $a_2 \notin K^2$.
    
    \item \label{2021_05_21_18:43}
    $F|K = K(x,w,z,y)|K$, where
    \[ w^2 = a_ 0 + x + a_2 x^2, \quad z^2 = b_1 w^2 + w, \quad y^2 = (c_3 + c_4 x + z) w, \]
    and $a_0,a_2,b_1,c_3,c_4 \in K$ are constants satisfying $a_2,b_1 \notin K^2$.

    \item \label{2021_05_21_18:45}
    $F|K = K(x,z,y)|K$, where    
    \[ z^4 = a_0 + x + a_2 x^2, \quad y^2 = c_0 + z + c_2 z^2, \]
    and $a_0,a_2,c_0,c_2 \in K$ are constants satisfying $a_2\notin K^2$ and $c_2 \in K^2 a_2$.
\end{enumerate}
In each case the singular prime $\pp$ is the only pole of the function $x$.
It has residue fields $\ka(\pp_3) = K$, $\ka(\pp_2) = K(a_2^{1/2})$,
and
\begin{enumerate}[\upshape (a)]
    \setcounter{enumi}{1}
    \item $\ka({\pp_1}) = K(a_2^{1/2}, b_2^{1/2})$, $\ka(\pp) = K(a_2^{1/2}, b_2^{1/4})$,
    \item $\ka({\pp_1}) = K(a_2^{1/2}, b_1^{1/2})$, $\ka(\pp) = K(a_2^{1/2}, b_1^{1/2}, (c_4 a_2^{1/2} + a_2 b_1^{1/2})^{1/2})$,
    \item $\ka(\pp_1) = \ka(\pp) = K(a_2^{1/2})$.
\end{enumerate}
\end{thm}

For the sake of clarity we avoid using the label ``(a)'', since in the next section this will stand for the function fields in Theorem~\ref{2024_02_18_00:35}.

\begin{rem}\label{2024_05_10_00:15}
Let $e\in\{1,2\}$ be the ramification index of the extension $\pp|\pp_1$. 
By \cite[Lemma~3.2]{HiSt23},
the assumption that $\pp$ is not a canonical divisor means that
\[ e=2 \quad \text{or} \quad \text{$\pp^e$ is not a canonical divisor}. \]
Suppose that $\pp^e$ is a canonical divisor. In other words, let $F|K$ be a geometrically rational function field of genus $g=3$ with a prime $\pp$ such that $\de(\pp)=3$, $\de(\pp_1)=1$, $\pp_2$ is non-rational, and such that the divisor $\pp^e$ is canonical.
Arguing as in the proof of \cite[Theorem~3.1~(ii)]{HiSt23}, by replacing $\pp$ with $\pp^e$ at each instance where $\pp$ is referred to as a divisor, and by removing the proof that ``$\pp|\pp_1$ is unramified'', 
one verifies that these assumptions mean that $F|K$ admits the following normal form
\begin{equation*}
    F|K=K(x,y)|K, \quad \text{where }z^4 = a_0 + x + a_2 x^2, \, y^2 = c_0 + c_1 x + z + c_2 z^2,
\end{equation*}
and $a_0,a_2,c_0,c_1,c_2 \in K$ are constants satisfying $a_2 \notin K^2$.
It also follows that $e=2$, i.e., $\pp$ is not a canonical divisor, if and only if
\begin{equation*}
    c_1=0 \quad \text{and} \quad c_2 \in K^2 + K^2 a_2,
\end{equation*}
in which case
we may normalize the coefficients in such a way that $c_2\in K^2 a_2$, thereby obtaining item~\ref{2021_05_21_18:45} in the theorem.
Consequently, item~\ref{2021_05_21_18:45} in Theorem~\ref{2021_05_21_18:40} occurs if and only if the power $\pp^e$ of the singular prime $\pp$ is a canonical divisor.
\end{rem}

\begin{proof}
In view of Remark~\ref{2024_05_10_00:15}, it suffices to treat the case where the divisor $\pp^e$ is non-canonical, where $e$ is the ramification index of $\pp|\pp_1$.
So let $F|K$ be a function field of genus $g = 3$ and let $\pp$ be a prime such that $\de(\pp) = 3$, $\de(\pp_1) = 1$, $\pp_2$ is non-rational, and such that the divisor $\pp^e$ is non-canonical. 
Since the restricted prime $\pp_3$ is rational, the prime $\pp_2$ is unramified over $F_3$ and has degree $\deg(\pp_2)=2$ .
We denote by $e_1$ the ramification index of $\pp_1|\pp_2$.

As explained in \cite[Remark~2.2]{HiSt23}, since the restricted prime $\pp_2$ is non-rational
the quasi-elliptic Frobenius pullback $F_1|K$ admits the following normal form
\begin{equation*}
    F_1|K=K(x,w,z)|K, \quad\text{where} \quad w^2 = a_0 + x + a_2 x^2, \, z^2 = b_0 + b_1 x +b_2 x^2 + w,
\end{equation*}
and $a_i,b_i \in K$ are constants satisfying $a_2\notin K^2$. 
The singular prime $\pp_1$ of $F_1|K$ is the only pole of the function $x$, and it has residue fields
\begin{equation}\label{2023_12_18_15:00}
    \ka(\pp_1)=K(a_2^{1/2},b_2^{1/2}), \quad \ka(\pp_2)=K(a_2^{1/2}), \quad \ka(\pp_3)=K.
\end{equation}
The functions $x$, $w$, $z$ are separating variables of the Frobenius pullbacks $F_3|K=K(x)|K$, $F_2|K=K(x,w)|K$, $F_1|K$ respectively,
and according to \cite[Remark~2.3]{HiSt23} they satisfy the incidence properties 
$x\in H^0(\pp_3) \setminus K$, 
$w\in H^0(\pp_2) \setminus H^0(\pp_3)$, 
$z\in H^0(\pp_1^{e_1}) \setminus H^0(\pp_2)$,
or more precisely
\begin{align*}
    H^0(\pp_3) &= K \oplus Kx, \quad
    H^0(\pp_2) = K \oplus Kx \oplus K w, \quad
    H^0(\pp_1^{e_1}) = K \oplus Kx \oplus K w \oplus K z.
\end{align*}
Moreover, since $\dim H^0(\pp_1^{2 e_1}) = \deg(\pp_1^{2 e_1}) = 8$ we also get
\begin{align}\label{2023_06_29_18:55}
    H^0(\pp_1^{2 e_1}) &= K \oplus Kx \oplus Kx^2 \oplus K w \oplus K x w \oplus K z \oplus K x z \oplus K w z.
\end{align}

By the Riemann--Roch theorem, since $F|K$ has genus $g = 3$ and the divisor $\pp^{ee_1}$ has degree $8$ one has
\[ \dim H^0 (\pp^{n e e_1}) = 8n - 2 \ \text{ for all $n\geq 1$.} \]
As is clear from $\dim H^0(\pp^{ee_1})=6 > \dim H^0(\pp_1^{e_1})=4$ and $H^0(\pp^{ee_1}) \cap F_1 = H^0(\pp_1^{e_1})$, there is a function $y$ in $H^0(\pp^{ee_1}) \setminus H^0(\pp_1^{e_1})$, which is a separating variable of $F|K$, that is, \[ F = F_1(y)=K(x,w,z,y). \]
Since its square $y^2$ lies in $H^0(\pp^{2ee_1}) \cap F_1 = H^0(\pp_1^{2e_1})$, but not in $F_2=K(x,w)$ as $y^2$ is a separating variable of $F_1|K$, there exist constants $c_i\in K$ with $(c_5, c_6, c_7)\neq(0,0,0)$ such that
\[ y^2 = c_0 + c_1 x + c_2 x^2 + (c_3 + c_4 x)w + (c_5 + c_6 x + c_7 w)z. \]

To study the singular prime $\pp$ of $F|K$ we introduce the functions $\breve{x} := x^{-1} \in F_3$, $\breve{w} := w x^{-1} \in F_2$, $\breve{z} := z x^{-1} \in F_1$ and $\cy := yx^{-1} \in F$, which satisfy the equations
\begin{align*}
    \breve{w}^2 &= a_2 + \breve{x} + a_0 \breve{x}^2, \quad \breve{z}^2 = b_2 + b_1 \breve{x} +b_0 \breve{x}^2 + \breve{x} \breve{w}, \\
    \cy^2 &= c_2 + c_1 \cx + c_0 \cx^2 + (c_4 + c_3 \cx)\breve{w} + (c_6 + c_5 \cx + c_7 \breve{w})\breve{z}.
\end{align*}
Note that $\cx$ is a local parameter at both $\pp_2$ and $\pp_3$, hence
\[ \breve{w}(\pp)^2 = a_2 \notin K^2, \quad \breve{z}(\pp)^2 = b_2, \quad \cy(\pp)^2 = c_2 + c_4 \breve{w}(\pp) + (c_6 + c_7 \breve{w}(\pp)) \breve{z}(\pp), \]
and
\[ \ka(\pp_3) = K, \quad \ka(\pp_2) = K(\breve{w}(\pp)), \quad \ka(\pp_1) = K(\breve{w}(\pp),\breve{z}(\pp)). \]
Furthermore, the differential $d\cy^8$ of $F_3|K=K(\cx)|K$ takes the form
\[ d\cy^8 = \big( ( c_6^4 + c_7^4 a_2^2) \cx^2 + c_7^4 \cx^4 + (c_5^4 + c_7^4 a_0^2)\cx^6 \big) d \cx.  \]

We claim that $c_6^2 + c_7^2 a_2 \neq 0$. 
Indeed, assume the contrary $c_6^2 + c_7^2 a_2 = 0$.
Since $a_2 \notin K^2$, this means that $c_6 = c_7 = 0$.
Then $c_5 \neq 0$ and so we may normalize $c_5 = 1$ by substituting $x$, $w$, $z$ with $c_5^{-4} x$, $c_5^{-2} w$, $c_5^{-1} z$ respectively. 
Since $v_{\pp_3}(d\cy^8) = v_{\pp_3}(\cx^6 d\breve{x}) = 6$, the value $\cy(\pp)$ of $\cy$ lies in $\ka({\pp_1}) = K(\breve{w}(\pp),\breve{z}(\pp))$, for otherwise $\pp$ is inertial over $F_1$ and $\de(\pp) = 2 \de(\pp_1) + \frac12 v_{\pp_3}(d\cy^8) = 5$ by \cite[Theorem~2.3]{BedSt87}, contradicting the assumption $\de(\pp)=3$. 
Thus $t(\pp) = 0$ for some $t$ in $\cy + K + K\breve{w} + K\breve{z} + K\breve{w}\breve{z}$, and it follows from
\[ v_{\pp_3}(dt^8) = v_{\pp_3}(d\cy^8) = 6 < 8 \]
that the prime $\pp_1$ is ramified over $F_2$, i.e., $e_1=2$, 
because otherwise $\pp$ is ramified over $F_1$ with local parameter $t$ and $\de(\pp) = 2 \de(\pp_1) + \frac12 v_{\pp_3}(dt^8) = 5$ by \cite[Theorem~2.3]{BedSt87}, a contradiction.
We infer that both $\breve{y}(\pp)$ and $\breve{z}(\pp)$ lie in $\ka(\pp_2)=K(\breve{w}(\pp))$,
and in turn $c_4=0$ as $\breve w(\pp)\notin K$.
Therefore $\breve{y}(\pp) = \al + \be \breve{w}(\pp)$ and $\breve{z}(\pp) = \te + \ga \breve{w}(\pp)$ for some $\al,\be,\te,\ga \in K$, 
hence by replacing $y$ and $z$ with $y + \al x + \be w$ and $z + \te x + \ga w$ respectively 
we may assume $\breve{z}(\pp)=\breve{y}(\pp)=0$, i.e., $b_2=c_2=0$.
Subtracting $b_0 + b_1 x$ from $w$ we can further normalize $b_0 = b_1 = 0$, i.e., $w = z^2$,
and so the divisor $\pp^e$ is canonical by Remark~\ref{2024_05_10_00:15}, a contradiction. This proves the claim.

It follows that $v_{\pp_3}(d\cy^8)=2$,
which by \cite[Theorem~2.3]{BedSt87} implies that $\de(\pp) = 2 \cdot 1 + \frac12 \cdot 2 = 3$ whenever $\cy(\pp)$ lies outside $\ka({\pp_1}) = K(\breve{w}(\pp),\breve{z}(\pp))$.
If this does not happen, say $t(\pp) = 0$ for some $t$ in $\cy + K + K\breve{w} + K\breve{z} + K\breve{w}\breve{z}$, then the prime $\pp$ is ramified over $F_1$ with local parameter $t$ since
\[ v_{\pp_3}(dt^8) = v_{\pp_3}(d\cy^8) = 2 < 4, \]
and hence $\de(\pp) = 2 \cdot 1 + \frac12 v_{\pp_3}(dt^8) = 3$ by \cite[Theorem~2.3]{BedSt87}. 
We have thus verified that our normal form already ensures that the hypothesis $\de(\pp)=3$ is satisfied. 
So it remains to study the assumptions that $F|K$ has genus $g = 3$ and that the divisor $\pp^e$ is non-canonical.
Note that the above also shows that $\ka(\pp) = K(\breve w(\pp),\breve z(\pp), \breve y(\pp))$.

By the Jacobian criterion and the genus drop formula, since $F|K=K(x,y)|K$
the assumption $g=3$ means that the zeros of the function
$\frac{dy^8}{dx} = (c_5 + c_6 x + c_7 w)^4$
are non-singular primes, that is, for every zero $\fq$ of the function
\[ 
a(x) := (c_5 + c_6 x)^2 + c_7^2 (a_0 + x + a_2 x^2 ) 
= (c_5^2 + c_7^2 a_0) + c_7^2 x + (c_6^2 + c_7^2 a_2)x^2 
\]
we have $\de(\fq)=0$. 
Note that $\de(\fq_1)=0$ holds already, since $F_1|K=K(x,w,z)|K$ is quasi-elliptic.
Two major cases are to be considered: $c_7 = 0$ and $c_7\neq 0$.
The former will correspond to item~\ref{2021_05_21_18:42} in the theorem, and the latter to item~\ref{2021_05_21_18:43}.

Assume first that $c_7 = 0$, so that $c_6 \neq 0$. 
One can then normalize $c_6 = 1$ and $c_5 = 0$ by replacing $x$, $w$, $z$, $y$ with 
$c_6^{-1} (c_6^{-3} x + c_5)$, $c_6^{-2}w$, 
$c_6^{-1} z$, $c_6^{-2} y$ respectively. 
Let $\fq$ be the only zero of $a(x) = x^2$, so that the function $x$ is a local parameter at the rational prime $\fq_3$ of $F_3|K = K(x)|K$. 
We want to see when $\de(\fq) = 0$ occurs. 
Since
\[ v_{\fq_3}(dy^8) = v_{\fq_3}(x^4 dx) = 4 > 0, \]
it is clear from \cite[Theorem~2.3]{BedSt87} that $\de(\fq)=0$ implies $y(\fq) \in \ka({\fq_1})$. 
Note that $\de(\fq)=0$ also implies $w(\fq) \in K$. 
Indeed, if $w(\fq) \notin K$
then the prime $\fq_1$ is unramified over $F_3$ with $\ka(\fq_1)=K(z(\fq))$ because $z(\fq)^4 = b_0^2 + w(\fq)^2 \notin K^2$, and in turn $y(\fq) \in \ka(\fq_1)$ yields $t(\fq) = 0$ for some $t$ in $y + K + Kz + Kz^2 + Kz^3$, whence $\fq$ is ramified over $F_1$ with local parameter $t$ as
\[ v_{\fq_3}(dt^8) = v_{\fq_3}(dy^8) = 4 < 8, \]
and therefore $\de(\fq) = \frac12 v_{\fq_3}(dt^8) = 2 > 0$ by \cite[Theorem~2.3]{BedSt87}.
Accordingly, we may assume $y(\fq) \in \ka(\fq_1)$ and $w(\fq) \in K$.
Subtracting from $w$ an element of $K$ we may further assume $w(\fq)=0$, that is, $a_0=0$ and $\fq_2$ is ramified (and therefore rational) over $F_3$ with local parameter $w$.
As the differential $dx$ of $F_2|K=K(x,w)|K$ vanishes, so that $v_{\fq_2} (d z^2) = v_{\fq_2}(d w) = 0$, applying \cite[Proposition~4.1]{BedSt87} we see that the prime $\fq_1$ is rational (and ramified over $F_2$) if and only if $z(\fq) = b_0^{1/2}$ belongs to $K$.
In particular $\ka({\fq_1}) = K(z(\fq))$, and so the condition $y(\fq) \in \ka({\fq_1})$ means that by replacing $y$ with an element of $y + K + Kz$ we can normalize $y(\fq) = 0$, i.e., $c_0 = 0$.

When $\fq_1$ is not rational, i.e., $z(\fq) = b_0^{1/2} \notin K$, 
we may assume $c_3 = 0$, for otherwise $\fq$ is ramified over $F_1$ with local parameter $y$ and $\de(\fq) = \frac12 v_{\fq_3}(dy^8) = 2 > 0$ by \cite[Theorem~2.3]{BedSt87}.
Now $v_{\fq_2}(y^4 + (c_1^2 + b_0) w^4)>4$ and $v_{\fq_2}(dy^4) = v_{\fq_2}(x^2 d w) = 4$, the latter due to the vanishing of the differential $dx$ of $F_2|K = K(x,w)$.
Since $(\frac{y}{w})(\fq)\notin \ka(\fq_1)=K(b_0^{1/2})$ as $(\frac{y}{w})(\fq)^4 = c_1^2 + b_0 \notin K^2$, 
this implies $\de(\fq) = \frac12 v_{\fq_2} \big( d(\frac yw)^4 \big) = 0$
by \cite[Theorem~2.3]{BedSt87}.

When $\fq_1$ is rational, i.e., $z(\fq)=b_0^{1/2} \in K$, we normalize $b_0 = 0$ by subtracting $b_0^{1/2}$ from $z$, and so $z$ is a local parameter at $\fq_1$. 
Since $v_{\fq_1}(y^2 + c_3 z^2) \geq 4$ and $v_{\fq_1}(dy^2) = v_{\fq_1}(x \, d z)=4$, the latter due to the vanishing of the differentials $dx$ and $dw$ of $F_1|K=K(x,w,z)|K$,
we conclude from \cite[Proposition~4.1]{BedSt87} that $\de(\fq)=0$ if and only if $c_3 \in K^2$, in which case we normalize $c_3=0$ by substituting $y$ with $y + c_3^{1/2} z$.

To sum up, in the first case $c_7=0$ the assumption $g=3$ has been translated into the normalizations $a_0=c_0=c_3=c_5=c_7=0$, $c_6=1$.
Replacing $z$ with $z + c_1 + c_2 x + c_4 w$ we normalize $c_1=c_2=c_4=0$, and by substituting $w$ with $w + b_1 x$ we normalize as well $b_1=0$. This yields the normal form in item~\ref{2021_05_21_18:42}.

 \medskip

Now we treat the second case $c_7 \neq 0$, where the polynomial $a(x)$ is separable. 
Replacing $x$, $w$, $z$, $y$ with $c_7^{-4} x$, 
$c_7^{-1} (c_7^{-1} w + c_5 + c_6 c_7^{-4} x)$, 
$c_7^{-1} z$, $c_7^{-1} y$ respectively we may normalize $c_7 = 1$, $c_5 = c_6 = 0$.
We claim that
\[ g = 3 \ \text{ if and only if } \  b(r),b(s),c(r),c(s) \in L^2, \]
where $b(x):= b_0 + b_1 x + b_2 x^2$, $c(x) := c_0 + c_1 x + c_2 x^2$, and $r,s \in L$ are the two roots of the polynomial 
$ a(x) = a_0 + x + a_2 x^2 $
in the separable closure $L$ of $K$.
By passing from $K$ to $L$ we may assume that $K$ is separably closed, i.e., $L=K$.
Then the function $a(x)$ has precisely two zeros (one for each root $r,s\in K$), and we must prove that for every such zero $\fq$ we have
\[ \de(\fq)=0 \ \text{ if and only if } \ z(\fq),y(\fq) \in K. \]
Let $r$ be the root of $a(x)$ corresponding to the zero $\fq$.
Since $x(\fq) = r \in K$, to see the claim one may assume $x(\fq) = 0$, that is, $a_0 = 0$ and $x$ is a local parameter at the rational prime $\fq_3$ of $F_3|K = K(x)|K$.
Then the prime $\fq_2$ is ramified (and therefore rational) over $F_3$ with local parameter $w$. 
As the differential $dx$ of $F_2|K = K(x,w)|K$ vanishes, hence
\[ v_{\fq_2}(dy^4) = v_{\fq_2}(w^2 dw) = 2 > 0, \]
it follows from \cite[Theorem~2.3]{BedSt87} that $y(\fq) \in \ka({\fq_1})$ whenever $\de(\fq)= 0$. 
Assuming that $z(\fq) \notin K$, the prime $\fq_1$ is unramified over $F_2$ with $\ka({\fq_1}) = K(z(\fq))$, and if we suppose $\de(\fq) = 0$ then $y(\fq) \in \ka({\fq_1})$ implies that $t(\fq) = 0$ for some $t$ in $K + Kz$, 
so that $\fq$ is ramified over $F_1$ with local parameter $t$ as
\[ v_{\fq_2}(dt^4) = v_{\fq_2}(dy^4) = 2 < 4, \]
and therefore $\de({\fq}) = \frac12 v_{\fq_2}(dt^4) = 1$, a contradiction. Thus $\de(\fq) = 0$ implies $z(\fq) \in K$. 
So in order to prove the claim we may suppose $z(\fq) = 0$, that is, $b_0 = 0$ and
$\fq_1$ is ramified (and hence rational) over $F_2$ with local parameter $z$.
Since $v_{\fq_1}(dy^2) = v_{\fq_1}(w dz) = 2$ as both differentials $dx$ and $dw$ of $F_1|K = K(x,w,z)|K$ vanish, we see from
\cite[Proposition~4.1]{BedSt87} that $\de(\fq) = 0$ if and only if $y(\fq) \in K$, thus proving the claim.

We next reformulate the above conditions $b(r),b(s),c(r),c(s) \in L^2$ as relations between the coefficients $b_i,c_i \in K$, by using symmetric polynomials.
Write
\[ q := r + s = a_2^{-1} \in K, \qquad t := rs = a_0 a_2^{-1} \in K, \]
and observe that the four symmetric polynomial expressions
\begin{align*}
	b(r) + b(s) &= b_1 q + b_2 q^2, \\
	r^2 b(r) + s^2 b(s) &= b_2 q^2 + b_1(q^3 + qt) + b_2 q^4, \\
	c(r) + c(s) &= c_1 q + c_2 q^2, \\
	r^2 c(r) + s^2 c(s) &= c_0 q^2 + c_1(q^3 + qt) + c_2 q^4,
\end{align*}
lie in $L^2 \cap K = K^2$, say they are equal to $\al^2$, $\be^2$, $\te^2$, $\ga^2$ respectively. Since $q \neq 0$ we can perform four normalizations as follows: 
replace $z$ with $z + \frac\al q x$, so that $b(r) + b(s) = 0$; 
substitute $z$ with $z + \frac\be q$, so that $r^2 b(r) + s^2 b(s) = 0$; 
replace $y$ with $y + \frac\te q x$, so that $c(r) + c(s) = 0$; 
substitute $y$ with $y + \frac\ga q$, so that $r^2 c(r) + s^2 c(s) = 0$. 
Now the four polynomial expressions vanish, 
i.e., $b(r) = b(s) = c(r) = c(s) = 0$, which means that $a(x)= a_0 + x + a_2 x^2$ divides both $b(x)$ and $c(x)$,
i.e.,
\[ z^2 = b_1 w^2 + w, \quad y^2 = c_1 w^2 + (c_3 + c_4 x) w + w z. \]
Replacing $z$ with $z + c_1 w$ we normalize $c_1 = 0$, thus obtaining a normal form as in item~\ref{2021_05_21_18:43} but without the condition $b_1 \notin K^2$.
To see that this requirement must indeed be part of our normal form we use the assumption that the divisor $\pp^e$ is non-canonical.
In view of Remark~\ref{2024_05_10_00:15} it suffices to observe that if $b_1\in K^2$ then by subtracting $b_1^{1/2} w$ from $z$ we get $w=z^2$ and $\big(\frac{y}{z}\big)^2 = c_3 + c_4 x + z + b_1^{1/2} z^2$.

To complete the proof 
it remains to verify that the normal forms in~\ref{2021_05_21_18:42} and~\ref{2021_05_21_18:43} guarantee that the divisor $\pp^e$ is non-canonical.
To this end we first find the space of global sections of $\pp^{ee_1}$. 
Since $\dim H^0(\pp^{ee_1})=6$ and $H^0(\pp_1^{e_1}) \oplus Ky \su H^0(\pp^{ee_1})$,
we must determine an element $u \in F$ such that
\begin{equation*}
    H^0(\pp^{ee_1}) = K \oplus Kx \oplus Kw \oplus Kz \oplus Ky \oplus Ku.
\end{equation*}
The function
\[ u:= 
\begin{cases}
	\frac{yw}x & \text{if item~\ref{2021_05_21_18:42},} \\
	\frac{yz}w & \text{if item~\ref{2021_05_21_18:43},}
\end{cases}
\]
fulfills this requirement, because the square
\[ u^2 = 
\begin{cases}
	(1 + a_2 x) z & \text{if item~\ref{2021_05_21_18:42},} \\
	(1 + b_1 w) ( c_3 + c_4 x + z) & \text{if item~\ref{2021_05_21_18:43},}
\end{cases}
\]
belongs to $H^0(\pp_1^{2 e_1}) = H^0(\pp^{2 e e_1}) \cap F_1$ (see~\eqref{2023_06_29_18:55}), i.e., $u \in H^0(\pp^{e e_1})$, while on the other hand the squares $1$, $x^2$, $w^2$, $z^2$, $y^2$, $u^2$ are linearly independent over $K^2$, i.e., $1$, $x$, $w$, $z$, $y$, $u$ are linearly independent over $K$.

We finally show that the divisor $\pp^e$ is non-canonical for both normal forms.
Seeking a contradiction let us assume that it is canonical, i.e.,
\[ \deg(\pp^e) = 2g-2 = 4 \quad \text{and} \quad \dim H^0(\pp^e) = g = 3. \]
As the divisor $\pp^{ee_1}$ has degree $8$, the condition $\deg(\pp^e) = 4$ rephrases as $e_1=2$, which means that the coefficient $b_2$ in our normal forms lies in $K^2(a_2)$  (see \eqref{2023_12_18_15:00}), say $b_2 = r_0^2 + r_1^2 a_2$.
(Note that for~\ref{2021_05_21_18:43} we have $b_2=a_2 b_1$.)
We check that in this case $\dim H^0(\pp^e) < 3$.
Since $H^0(\pp^e)$ is contained in $H^0(\pp^{ee_1})$, any element $h \in H^0(\pp^e)$ can be written as $h = d_1  + d_2 x + d_3 w + d_4 z + d_5 y + d_6 u $.
Moreover its square $h^2$ belongs to $H^0({\pp^{2e}})\cap F_1 = H^0(\pp_1^2) = K \oplus Kx \oplus Kw \oplus Kz$.
Using the condition $a_2\notin K^2$ (and $b_1\notin K^2$ for~\ref{2021_05_21_18:43}), we obtain $d_5=d_6 = 0$, $d_2 = d_4 r_0$, $d_3=d_4 r_1$, and therefore $h\in K \oplus K (r_0 x + r_1 w + z)$, as desired.
\end{proof}

\begin{rem}\label{2024_05_10_01:35}
Let $F|K=K(x,w,z,y)|K$ be a function field as in Theorem~\ref{2021_05_21_18:40}, item~\ref{2021_05_21_18:42} or~\ref{2021_05_21_18:43}.
As the proof of the theorem shows,
the first Frobenius pullback $F_1|K=K(x,w,z)|K$ is quasi-elliptic, the second Frobenius pullback $F_2|K=K(x,w)|K$ has genus $g_2=0$, and the third Frobenius pullback $F_3|K=K(x)|K$ is rational.
Moreover $e_2=1$, $\divv_\infty(x) = \pp^{e e_1}$, $\deg(\pp_1)=4/e_1$, and $\deg(\pp)=8/ee_1$, where $e$, $e_1$ and $e_2$ are the ramification indices of $\pp$, $\pp_1$ and $\pp_2$ over $F$, $F_1$ and $F_2$ respectively.
Furthermore,
\begin{align*}
    H^0(\pp_3) &= K \oplus Kx, \\
    H^0(\pp_2) &= K \oplus Kx \oplus K w, \\
    H^0(\pp_1^{e_1}) &= K \oplus Kx \oplus K w \oplus K z, \\
    H^0(\pp^{ee_1}) &= K \oplus Kx \oplus Kw \oplus Kz \oplus Ky \oplus Ku,
\end{align*}
where 
\[ u:= 
\begin{cases}
	\frac{yw}x & \text{if item~\ref{2021_05_21_18:42},} \\
	\frac{yz}w & \text{if item~\ref{2021_05_21_18:43}.}
\end{cases}
\]
Depending on the value of $e_1 \in \{1,2\}$, which in turn depends on whether 
$b_2 \in K^2(a_2)$ (if item~\ref{2021_05_21_18:42}) or $b_1 a_2 \in K^2(a_2)$ (if item~\ref{2021_05_21_18:43}),
the quasi-elliptic Frobenius pullback $F_1|K$ can be of type (ii) or (iii) in \cite[Theorem~2.1]{HiSt23}.
\end{rem}


\begin{prop}\label{2024_05_10_01:40}
    Let $F|K$ and $F'|K$ be two function fields as in Theorem~\ref{2021_05_21_18:40}.
    
    \textnormal{(b)}
    Assume that $F|K$ and $F'|K$ are of type~\ref{2021_05_21_18:42}.
    Then $F|K$ and $F'|K$ are isomorphic if and only if there exist constants $\mu_1,\mu_2,\mu_4,\mu_5 \in K$ satisfying $(\mu_4,\mu_5) \neq (0,0)$ and
    \begin{align*}
        \eps^4 a_2' &= a_2, \qquad
        \eps^6 b_2' = b_2 + \ga^2, \qquad
        b_0' = \eps^2 b_0 + \eps \mu_4 \mu_5 + \mu_2^4 + \mu_5^4 b_2,
    \end{align*}
    where $\eps := \mu_4^2 + \mu_5^2 a_2 \neq 0$ and $\ga := \eps^{-1} (\mu_1^2 + \mu_2^2 a_2)$.    
    The corresponding $K$-isomorphisms $F'\overset\sim\to F$ are given by the transformations
    \[ (x',w',z',y') \mapsto ( \eps^3 (\mu_5^2 + \eps x), \eps(\mu_4 \mu_5 + \eps w), \eps (\be + \ga x + z), \eps^2 (\tau + \mu_1 x + \mu_2 w + \mu_4 y + \mu_5 u)), \]
    where $\be := \eps^{-1} (\ga \mu_5^2 + \mu_2^2)$ and $\tau := \eps^{-1} (\mu_1 \mu_5^2 + \mu_2 \mu_4 \mu_5)$.
    In particular, the group $\mathrm{Aut}(F|K)$ of automorphism of $F|K$ is trivial.
    
    \textnormal{(c)}
    Assume that $F|K$ and $F'|K$ are of type~\ref{2021_05_21_18:43}.
    Then $F|K$ and $F'|K$ are isomorphic if and only if there exist constants $\al,\mu_2,\mu_3,\mu_4,\mu_5 \in K$ satisfying $(\mu_4,\mu_5) \neq (0,0)$ and
    \begin{align*}
       \eps^4 a_2' &= a_2, \qquad
       \eps^{-2} a_0' = \eps^2 (a_0 + \al^2 a_2 + \al) + \mu_5^4, \qquad
       \eps^3 c_4' = c_4, \\
       \eps^2 b_1' &= b_1 + \ga^2, \qquad
       c_3' = \eps (c_3 + \al c_4) + \mu_3^2 + \mu_4 \mu_5,
    \end{align*}
    where $\eps := \mu_4^2 + \mu_5^2 b_1 \neq 0$ and
    $\ga := \eps^{-1} (\mu_2^2 + \mu_3^2 b_1)$.
    The corresponding $K$-isomorphisms $F'\overset\sim\to F$ are given by the transformations
    \[ (x',w',z',y') \mapsto ( \eps^4 (\al + x), \eps(\mu_5^2 + \eps w) ,\eps (\beta + \ga w + z), \eps (\tau + \mu_2 w + \mu_3 z + \mu_4 y + \mu_5 u)), \]
    where $\beta :=\eps^{-1} (\ga \mu_5^2 + \mu_4 \mu_5)$ and
    $\tau :=\eps^{-1} (\mu_2 \mu_5^2 + \mu_3 \mu_4 \mu_5)$.
    In particular, the quotient $\iota := c_4^4/a_2^3$ is an invariant of $F|K$.
    If $\iota\neq0$ then the group $\mathrm{Aut}(F|K)$ of automorphisms of $F|K$ is trivial. If $\iota=0$ then the group $\mathrm{Aut}(F|K)$ is isomorphic to $\Z/2\Z$ and it is generated by the transformation
    \[ (x,w,z,y) \mapsto (x + a_2^{-1}, w, z, y). \]
\end{prop}

The proposition can be proved by using the incidence properties of the functions $x,w,z,y$, as in the proof of Proposition~\ref{2024_03_25_19:55}. 
We will check in the next section (see Corollary~\ref{2024_08_25_19:45}) that function fields of type~\ref{2021_05_21_18:42} and~\ref{2021_05_21_18:43} are indeed non-isomorphic.

\begin{rem}\label{2024_05_10_20:05}
For the function fields in Theorem~\ref{2021_05_21_18:40}~\ref{2021_05_21_18:45} there are statements similar to those in Remark~\ref{2024_05_10_01:35} and Proposition~\ref{2024_05_10_01:40}, which come from analogous assertions in \cite[Remark~3.4 and Proposition~3.5]{HiSt23} (see Remark~\ref{2024_05_10_00:15}). 
The statement in \cite[Proposition~3.5~(ii)]{HiSt23} is valid here with no modifications.
In \cite[Remark~3.4]{HiSt23} the following alterations must be made in the discussion about the function fields in \cite[Theorem~3.1~(ii)]{HiSt23}: $e=2$, $\mathrm{div}_\infty(x) = \pp^{2e} = \pp^4$, and $H^0(\pp^e) = H^0(\pp^2) = K \oplus Kz \oplus Ky$.
\end{rem}

\begin{rem}\label{2024_08_25_19:50}
As follows from \cite[Section~3]{HiSt23} and the results in this section, the one-dimensional separable function fields $F|K$ of genera $g=3$, $g_1=1$, $g_2=\ov g = 0$ in characteristic $p=2$ can be grouped into five classes
\begin{enumerate}[\upshape (i)]
    \item \label{2024_08_25_19:51}
    function fields in \cite[Theorem~3.1~(i)]{HiSt23},
    \item \label{2024_08_25_19:52}
    function fields in \cite[Theorem~3.1~(ii)]{HiSt23} and Theorem~\ref{2021_05_21_18:40}~\ref{2021_05_21_18:45},
    \item \label{2024_08_25_19:53}
    function fields in Theorem~\ref{2024_02_18_00:35},
    \item \label{2024_08_25_19:54}
    function fields in Theorem~\ref{2021_05_21_18:40}~\ref{2021_05_21_18:42},
    \item \label{2024_08_25_19:55}
    function fields in Theorem~\ref{2021_05_21_18:40}~\ref{2021_05_21_18:43},
\end{enumerate}
that are distinguished by the properties
\begin{enumerate}[(i)]
    \item the prime $\pp_2$ is rational and the divisor $\pp$ is canonical,
    \item the prime $\pp_2$ is non-rational and the divisor $\pp^e$ is canonical,
    \item the prime $\pp_2$ is rational and the divisor $\pp$ is non-canonical,
    \item[(iv),(v)] the prime $\pp_2$ is non-rational and the divisor $\pp^e$ is non-canonical.
\end{enumerate}
Here $\pp$ is the only singular prime of $F|K$ and $e\in\{1,2\}$ is the ramification index of $\pp|\pp_1$.
Note that the divisor $\pp$ is canonical if and only if the divisor $\pp^e$ is canonical and $e=1$ (see \cite[Lemma~3.2]{HiSt23}).
Note also that $e=1$ in cases~\ref{2024_08_25_19:51} and~\ref{2024_08_25_19:53}, and that in case~\ref{2024_08_25_19:52} we have $e=1$ (resp. $e=2$) if $F|K$ satisfies \cite[Theorem~3.1~(ii)]{HiSt23} (resp. Theorem~\ref{2021_05_21_18:40}~\ref{2021_05_21_18:45}).
The quasi-elliptic Frobenius pullbacks $F_1|K$ of $F|K$ are of type
\begin{itemize}
    \item \cite[Theorem~2.1~(i)]{HiSt23} in cases~\ref{2024_08_25_19:51} and~\ref{2024_08_25_19:53}, 
    \item \cite[Theorem~2.1~(iii)]{HiSt23} in case~\ref{2024_08_25_19:52}, 
    \item \cite[Theorem~2.1, (ii) or (iii)]{HiSt23} in cases~\ref{2024_08_25_19:54} and~\ref{2024_08_25_19:55};
\end{itemize}
see \cite[Remark~3.4]{HiSt23} and Remarks~\ref{2024_08_26_14:35}, \ref{2024_05_10_01:35} and~\ref{2024_05_10_20:05}.

In the cases~\ref{2024_08_25_19:52}, \ref{2024_08_25_19:53} and~\ref{2024_08_25_19:55} the expressions $\iota = c_1^4 / a_2^3$, $\iota = c_1 B_1^2$ and $\iota = c_4^4/a_2^3$ are invariants of $F|K$ respectively; 
see \cite[Proposition~3.5]{HiSt23}, Corollary~\ref{2024_03_25_11:50}, Proposition~\ref{2024_05_10_01:40} and Remark~\ref{2024_05_10_20:05}.
As $\iota$ can take any value in $K$ (for~\ref{2024_08_25_19:53}) or $K \setminus K^2$ (for~\ref{2024_08_25_19:52} and~\ref{2024_08_25_19:55}) we deduce that there are infinitely many isomorphism classes of function fields in~\ref{2024_08_25_19:52}, \ref{2024_08_25_19:53} and~\ref{2024_08_25_19:55}.
Moreover, it will follow from the next section (see also \cite[Section~4]{HiSt23}) that in these three cases~\ref{2024_08_25_19:52}, \ref{2024_08_25_19:53} and~\ref{2024_08_25_19:55} the function field $F|K$ is non-hyperelliptic if and only if the invariant $\iota$ does not vanish, while in the remaining two cases~\ref{2024_08_25_19:51} and~\ref{2024_08_25_19:54}, where we do not have an invariant, the function field $F|K$ is always non-hyperelliptic.
\end{rem}

Let us finish this section by noting that the hypotheses $g=3$, $g_1=1$, $g_2=0$ already imply that the characteristic $p$ is equal to $2$ (see \cite[Corollary~2.7~(iii)]{HiSt22}).
Thus the condition $p=2$ in Remark~\ref{2024_08_25_19:50} can be removed. 
It also follows (see the beginning of this section) that we can remove the hypothesis $p=2$ in Theorems~\ref{2024_02_18_00:35} and~\ref{2021_05_21_18:40}.

\section{Regular but non-smooth plane quartic curves}
\label{2024_04_05_18:30}

In this section we study the function fields $F|K$ of the preceding section that are non-hyperelliptic. 
We realize their regular proper models as plane quartic curves in $\PP^2(K)$, via canonical embeddings.

Recall that the \emph{canonical field} of a one-dimensional function field $F|K$ is the subfield generated by the quotients of the non-zero holomorphic differentials of $F|K$. 
Equivalently, the canonical field of $F|K$ is the subfield generated over $K$ by the global sections of any canonical effective divisor.
A function field $F|K$ of genus $g\geq 2$ is called \emph{hyperelliptic} if it admits a quadratic subfield of genus zero;
this subfield is uniquely determined because it coincides with the canonical field of $F|K$.
The function field $F|K$ is non-hyperelliptic if and only if its canonical field coincides with the entire function field.

\begin{prop}
Every function field $F|K$ in Theorem~\ref{2021_05_21_18:40}~\ref{2021_05_21_18:45} is hyperelliptic.
\end{prop}

\begin{proof}
Since $z^4 = a_0 + x + a_2 x^2$ and $a_2 \neq 0$ the degree of the extension $K(z,y) \su F=K(z,y,x)$ is equal to $1$ or $2$. 
The subfield $K(z,y)$ has genus zero because its generators satisfy the quadratic relation $y^2 = c_0 + z + c_2 z^2$, and so $F \neq K(z,y)$.
\end{proof}

We wish to identify which function fields in Theorem~\ref{2024_02_18_00:35} and Theorem~\ref{2021_05_21_18:40}, \ref{2021_05_21_18:42} and~\ref{2021_05_21_18:43}, are non-hyperelliptic.
Unlike our analysis in \cite[Section~4]{HiSt23}, here we have no a priori knowledge of canonical divisors, nor of their global sections. 
However, we can still work with the spaces of holomorphic differentials.

\begin{lem}\label{2024_10_16_15:05}
Let $F|K$ be a function field as in Theorem~\ref{2024_02_18_00:35}, Theorem~\ref{2021_05_21_18:40}~\ref{2021_05_21_18:42} or Theorem~\ref{2021_05_21_18:40}~\ref{2021_05_21_18:43}.
Then the divisor of the differential $dy$ is given by
\begin{enumerate}[\upshape (a)]
    \item $\mathrm{div}(dy) = \mathrm{div}_0 (c_0 + c_1 x + x^2 )^{1/2}$ in Theorem~\ref{2024_02_18_00:35},
    \item $\mathrm{div}(dy) = \mathrm{div}_0(x)^{1/2}$ in Theorem~\ref{2021_05_21_18:40}~\ref{2021_05_21_18:42}, 
    \item $\mathrm{div}(dy) = \mathrm{div}_0(w)^{1/2}$ in Theorem~\ref{2021_05_21_18:40}~\ref{2021_05_21_18:43}.
\end{enumerate}
\end{lem}

\begin{proof}
By passing to the separable closure of $K$ we may assume that $K$ is separably closed.
For each prime $\fq $ of $F|K$ different from the only singular prime $\pp$ we have 
\[ \deg(\fq) \cdot v_\fq(dy) = \deg(\fq_2) \cdot v_{\fq_2} (dy^4) = \deg(\fq_3) \cdot v_{\fq_3} (dy^8) \] 
(see \cite[Theorem~2.7]{BedSt87}).
If $F|K$ is given as in Theorem~\ref{2024_02_18_00:35}, then
the differential $dy^4$ of $F_2|K=K(x)|K$ is equal to $(c_0 + c_1 x + x^2)^2 dx$, so it follows that at the primes $\fq$ different from the pole $\pp$ of the function $x$ the divisor of $dy$ is equal to the divisor $\mathrm{div}_0 ( c_0 + c_1 x + x^2 )^{1/2}$.
If $F|K$ satisfies Theorem~\ref{2021_05_21_18:40}~\ref{2021_05_21_18:42}
(resp. Theorem~\ref{2021_05_21_18:40}~\ref{2021_05_21_18:43}),
then the differential $dy^8$ of $F_3|K = K(x)|K$ is equal to $x^4 dx$ (resp. $w^4 dx = (a_0 + x + a_2 x^2)^2 dx$), hence we conclude that at the primes $\fq$ different from the pole $\pp$ of $x$ the divisor of $dy$ is equal to the divisor $\mathrm{div}_0(x)^{1/2}$ (resp. $\mathrm{div}_0(w)^{1/2}$).

It remains to note that by \cite[Theorem~2.7]{BedSt87},
at the singular prime $\pp$ the product $\deg(\pp) \cdot v_\pp(dy)$ is equal to $2 \de(\pp) + v_{\pp_2}(dy^4)$ if $F|K$ satisfies Theorem~\ref{2024_02_18_00:35}, or to $2 \de(\pp) + v_{\pp_3}(dy^8)$ if $F|K$ satisfies Theorem~\ref{2021_05_21_18:40}, \ref{2021_05_21_18:42} or~\ref{2021_05_21_18:43}, that is, $v_\pp(dy)=0$.
\end{proof}

\begin{prop}\label{2024_08_04_19:35}
Let $F|K$ be a function field as in Theorem~\ref{2024_02_18_00:35}, Theorem~\ref{2021_05_21_18:40}~\ref{2021_05_21_18:42} or Theorem~\ref{2021_05_21_18:40}~\ref{2021_05_21_18:43}.
Then a basis of the vector space of holomorphic differentials of $F|K$ is given by
\begin{enumerate}[\upshape (a)]
    \item $dy, \, \breve z \, dy, \, \breve y \, dy$, where $\breve z:= z \cdot c(x)^{-1}$, $\breve y := y \cdot c(x)^{-1}$, $c(x) := c_0 + c_1 x + x^2$, in Theorem~\ref{2024_02_18_00:35};
    \item $dy$, $\breve w \, dy$, $\breve y \, dy$, where $\breve w := wx^{-1}$, $\breve y := y x^{-1}$, in Theorem~\ref{2021_05_21_18:40}~\ref{2021_05_21_18:42};
    \item $dy$, $\breve z \, dy$, $\breve y \, dy$, where $\breve z := zw^{-1}$, $\breve y := yw^{-1}$, in Theorem~\ref{2021_05_21_18:40}~\ref{2021_05_21_18:43}.
\end{enumerate}
In each case the first two differentials form a basis of the vector space of exact holomorphic differentials of $F|K$.
\end{prop}

\begin{proof}
Suppose $F|K$ is given as in Theorem~\ref{2024_02_18_00:35}.
Since the functions $\breve y$ and $\breve z$ fulfill the relations $\breve z^2 = A_2 + c_1^{-1} \cdot c(x)^{-1}$ and $\breve y^2 = (B_0 + B_1 x) \cdot c(x)^{-1} + \breve z$,
their pole divisors satisfy $\mathrm{div}_\infty(\breve z) = \divv_0(c(x))^{1/2}$ and $\mathrm{div}_\infty(\breve y) \leq \divv_0(c(x))^{1/2}$.
By the preceding lemma, this implies that the three differentials are holomorphic. As they are linearly independent and as their number is equal to the dimension $g=3$ of the vector space of holomorphic differentials, they form a basis of this vector space.

For $F|K$ as in Theorem~\ref{2021_05_21_18:40}, \ref{2021_05_21_18:42} or~\ref{2021_05_21_18:43}, we argue analogously.
If $F|K$ is of type~\ref{2021_05_21_18:42} (resp. type~\ref{2021_05_21_18:43}), then the relations
$\breve w^2 = a_2 + x^{-1}$ and $\breve y^4 = b_2 + b_0 x^{-2} + \breve w x^{-1}$
(resp. $\breve z^2 = b_1 + w^{-1}$, $w^2 = a_0 + x + a_2 x^2$ and $\breve y^2 = c_3 w^{-1} + c_4 x w^{-1} + \breve z$)
show that
$\divv_\infty(\breve w) = \divv_0(x)^{1/2}$ and $\divv_\infty(\breve y) \leq \divv_0(x)^{1/2}$
(resp. $\divv_\infty (\breve z) = \divv_0(w)^{1/2}$ and $\divv_\infty(\breve y)\leq \divv_0(w)^{1/2}$),
that is, the three differentials are holomorphic and form a basis of the vector space of holomorphic differentials of $F|K$.

It remains to note that as $F=F_1 \oplus F_1 y$, the space of exact differentials of $F|K$ is equal to $F_1 \, dy$.
Hence in each case the first two elements of the basis are exact, while the third element is not exact.
\end{proof}

We denote by $H$ the \textit{canonical field} of $F|K$, i.e., the subfield of $F|K$ generated by the quotients of the non-zero holomorphic differentials of $F|K$.
Furthermore, we denote by $E$ the subfield of $H$ generated by the quotients of the non-zero \textit{exact} holomorphic differentials of $F|K$, and call it the \textit{pseudo-canonical field} of $F|K$.
Due to our description of the bases in the preceding Proposition, we obtain

\begin{cor} $ $
\begin{enumerate}[\upshape (a)]
    \item $H = K(\breve z, \breve y)$, $E=K(\breve z)$.
    \item $H=K(\breve w,\breve y)$, $E=K(\breve w)$.
    \item $H=K(\breve z,\breve y)$, $E=K(\breve z)$.
\end{enumerate}
\end{cor}




\begin{cor}\label{2024_08_25_19:45}
The pseudo-canonical field $E$ of $F|K$ is a rational quadratic subfield of the Frobenius pullback $F_1$.
More precisely, depending on whether $F|K$ is given as in Theorem~\ref{2024_02_18_00:35}, Theorem~\ref{2021_05_21_18:40}~\ref{2021_05_21_18:42} or Theorem~\ref{2021_05_21_18:40}~\ref{2021_05_21_18:43}, the following holds
\begin{enumerate}[\upshape (a)]
    \item the quadratic extension $F_1|E$ is separable and the function $x$ is a generator,
    \item the quadratic extension $F_1|E$ is inseparable and $E$ is equal to the second Frobenius pullback $F_2$,
    \item the quadratic extension $F_1|E$ is separable and the function $x$ is a generator.
\end{enumerate}
\end{cor}

In particular $E\neq H$, because $[F:E]=4$ and $[F:H] \leq 2$.

\begin{proof}
By the preceding corollary the pseudo-canonical field $E$ is a rational subfield of the Frobenius pullback $F_1$.
It is properly contained in $F_1$ because $F_1|K$ has genus $g_1=1\neq 0$.

If $F|K$ in given as in Theorem~\ref{2024_02_18_00:35} then $E(x) = K(\breve z,x) =  K(x,z)=F_1$, where $x$ satisfies over $E$ the quadratic separable equation $x^2 + c_1 x + c_2 + c_1^{-1}(A_2 + \breve z^2)^{-1} = 0$.
For $F|K$ as in Theorem~\ref{2021_05_21_18:40}~\ref{2021_05_21_18:42} we have $x = (a_2 + \breve w^2)^{-1}$, and so $E=K(\breve w) = K(x,w) = F_2$.
Finally, for $F|K$ as in Theorem~\ref{2021_05_21_18:40}~\ref{2021_05_21_18:43} we see that 
$ w = (\breve z^2 + b_1)^{-1}$, hence $E=K(\breve z) = K(w,z)$ and $E(x) = K(x,w,z) = F_1$, where $x$ satisfies over $E$ the quadratic separable equation $a_2 x^2 + x + a_0 + w^2 = 0$.
\end{proof}

\begin{cor}\label{2022_03_11_14:35}
The function field $F|K$ is given as in Theorem~\ref{2021_05_21_18:40}~\ref{2021_05_21_18:42}  
if and only if the pseudo-canonical field $E$ is equal to the second Frobenius pullback $F_2$. 
\end{cor}

We have thus obtained a conceptual characterization of the three types of function fields.
Indeed, the function fields in Theorem~\ref{2021_05_21_18:40}~\ref{2021_05_21_18:42} are distinguished by the above corollary, while the function fields in Theorem~\ref{2024_02_18_00:35} are characterized by the property that the restricted prime $\pp_2$ is rational.

In the corollary below we employ the notations of Propositions~\ref{2024_03_25_19:55} and~\ref{2024_05_10_01:40}.

\begin{cor}\label{2024_08_04_20:25}
If $F|K$ and $F'|K$ are two function fields as in Theorem~\ref{2024_02_18_00:35}, Theorem~\ref{2021_05_21_18:40}~\ref{2021_05_21_18:42}, or Theorem~\ref{2021_05_21_18:40}~\ref{2021_05_21_18:43},
then each $K$-isomorphism $F' \overset\sim\to F$ defined by the constants 
\begin{enumerate}[\upshape (a)]
    \item $\al,\mu_2,\mu_3,\mu_4,\mu_5 \in K$,
    \item $\mu_1,\mu_2,\mu_4,\mu_5 \in K$,
    \item $\al,\mu_2,\mu_3,\mu_4,\mu_5 \in K$,
\end{enumerate}
in Propositions~\ref{2024_03_25_19:55} and~\ref{2024_05_10_01:40} induces $K$-isomorphisms $E' \overset\sim\to E$ and $H' \overset\sim \to H$ according to the rules
\begin{enumerate}[\upshape (a)]
    \item $\displaystyle\breve z' \mapsto \frac 1 {\eps^3} \bigg( \ga + \frac{\mu_5 A_2 + \mu_4 \breve z}{\mu_4 + \mu_5 \breve z} \bigg), \quad 
    \breve y' \mapsto \frac 1 {\eps^2} \bigg( \mu_2 + \frac{\mu_3 (\mu_5 A_2 + \mu_4 \breve z) + \eps \breve y}{\mu_4 + \mu_5 \breve z} \bigg)$;
 
    \item $\displaystyle \breve w' \mapsto \frac 1 {\eps^2} \bigg( \frac{\mu_5 a_2 + \mu_4 \breve w}{\mu_4 + \mu_5 \breve w}\bigg), \quad
    \breve y' \mapsto \frac 1 {\eps^2} \bigg( \mu_1 + \frac{\mu_2 (\mu_5 a_2 + \mu_4 \breve w) + \eps \breve y}{\mu_4 + \mu_5 \breve w} \bigg)$;
    
    \item $\displaystyle  \breve z' \mapsto \frac 1 \eps \bigg( \ga + \frac{\mu_5 b_1 + \mu_4 \breve z}{\mu_4 + \mu_5 \breve z}\bigg), \quad
    \breve y' \mapsto \frac 1 \eps \bigg( \mu_2 + \frac{\mu_3 (\mu_5 b_1 + \mu_4 \breve z) + \eps \breve y}{\mu_4 + \mu_5 \breve z} \bigg)$.
\end{enumerate}
\end{cor}

\begin{prop}\label{2024_08_04_19:50}
$ $
\begin{enumerate}[\upshape (a)]
    \item A function field $F|K$ as in Theorem~\ref{2024_02_18_00:35} is non-hyperelliptic if and only if $B_1 \neq 0$, that is, if and only if the invariant $\iota = c_1 B_1^2$ does not vanish.
    
    \item Every function field $F|K$ in Theorem~\ref{2021_05_21_18:40}~\ref{2021_05_21_18:42} is non-hyperelliptic.
    
    \item A function field $F|K$ as in Theorem~\ref{2021_05_21_18:40}~\ref{2021_05_21_18:43} is non-hyperelliptic if and only if $c_4\neq 0$, that is, if and only if the invariant $\iota = c_4^4/a_2^3$
does not vanish.
\end{enumerate}
\end{prop}

\begin{proof}
Let $F|K$ be given as in Theorem~\ref{2024_02_18_00:35}. 
Then $F=K(x,z,y)$ has degree $\leq 2$ over the canonical field $H=K(\breve z, \breve y)$, since $F=H(x)$ and $x^2 + c_1 x + c_2 + c_1^{-1}(A_2 + \breve z^2)^{-1} = 0$.
The equation $\breve y^2 = c_1 (B_0 + B_1 x)(A_2 + \breve z^2) + \breve z$ shows that $H$ contains the function $x$, provided that $B_1\neq0$.
It also shows that if $B_1=0$ then $H$ has genus zero and is purely inseparable of degree $\leq 2$ over the pseudo-canonical field $E=K(\breve z)$, 
that is, $H$ is a quadratic subfield of genus zero of $F|K$.

Suppose that $F|K$ is given as in Theorem~\ref{2021_05_21_18:40}~\ref{2021_05_21_18:42}. 
As $x\in F_2=E\subset H = K(\breve w,\breve y)$, the functions $x$, $w = x \breve{w}$, $y=x \cy$ and $z = y^2 x^{-1}$ belong to $H$ and therefore $F=H$.

Assume next that $F|K$ is given as in Theorem~\ref{2021_05_21_18:40}~\ref{2021_05_21_18:43}.
Then $w=(b_1 + \breve z^2)^{-1}$ and $\breve y^2 = (c_3 + c_4 x) (b_1 + \breve z^2) + \breve z$.
It follows that if $c_4\neq0$ then the canonical field $H=K(\breve z,\breve y)$ contains the functions $x$, $w$, $z=w \breve z$ and $y = w \breve y$, that is, $H=F$.
It also follows that if $c_4=0$ then $H$ has genus zero and is inseparable of degree $2$ over $E=K(\breve z)$, that is, $H$ is a quadratic subfield of genus zero of $F|K$.
\end{proof}

Now we use the sections of the canonical divisor $\mathrm{div}(dy)$ (see Lemma~\ref{2024_10_16_15:05}) to realize the non-hyperelliptic function fields $F|K$ in the above proposition as curves of degree $2g-2=4$ in $\PP^{g-1}(K) = \PP^2$, where $g=3$ is the genus of $F|K$.

\begin{thm}\label{2024_05_27_00:10}
The one-dimensional separable non-hyperelliptic geometrically rational function fields $F|K$ of characteristic $p=2$ and genus $g=3$ admitting a singular prime $\pp$ that is not a canonical divisor are classified as follows:
\begin{enumerate}[\upshape (a)]
    \item \label{2024_05_27_00:11}
    If $\pp_2$ is rational then $F|K$ is the function field of a regular plane projective integral quartic curve over $\Spec K$ with generic point $(1: y : z)$ that satisfies the quartic equation
    \[ b y^4 + d z^4 + y^2 z^2 + z^3 + (b + b^2 c^3) z^2 + a y^2 + a z + a b^2 c^3 + a^2 d = 0, \]
    where $a,b,c,d\in K$ are constants satisfying $a\notin K^2$ and $b,c\neq 0$.
    The singular prime $\pp$ is centered at the point $(1:a^{1/4}:a^{1/2})$ in $\PP^2(\ov K)$ and has residue fields $\ka(\pp)=K(a^{1/4})$, $\ka(\pp_1)=K(a^{1/2})$, $\ka(\pp_2) = K$.
    
    \item \label{2024_05_27_00:12}
    If $\pp_2$ is non-rational and the pseudocanonical field $E$ is equal to the second Frobenius pullback $F_2$, then $F|K$ is the function field of a regular plane projective integral quartic curve over $\Spec K$ with generic point $(1:y:z)$ that satisfies the quartic equation
    \[ y^4 + a z^4 + z^3 + b z + c = 0, \]
    where $a,b,c \in K$ are constants satisfying $b \notin K^2$.
    The singular prime $\pp$ is centered at the point $(1:(a b^2 + c)^{1/4} : b^{1/2})$ in $\PP^2(\ov K)$ and has residue fields $\ka(\pp)=K(b^{1/2}, (a b^2 + c)^{1/4})$, $\ka(\pp_1)=K(b^{1/2}, (a b^2 + c)^{1/2})$, $\ka(\pp_2) = K(b^{1/2})$, $\ka(\pp_3)=K$.
    
    \item \label{2024_05_27_00:13}
    If $\pp_2$ is non-rational and the pseudocanonical field $E$ is not equal to the second Frobenius pullback $F_2$,
    then $F|K$ is the function field of the regular regular plane projective geometrically integral curve over $\Spec K$ with the generic point $(1:y:z)$ which satisfies the homogeneous quartic equation
    \[ y^4 + d z^2 y^2 + (c+a) z^4 + d z^3 + bd\, y^2 + z^2 + bd\, z + b^2 c = 0, \]
    where $a,b,c,d\in K$ satisfy $a,b\notin K^2$ and $d\neq 0$.
    The prime $\pp$ is centered at $(1:(a b^2 + b)^{1/4}:b^{1/2})$, and has residue fields $\ka(\pp) = K(a^{1/2}, b^{1/2}, (a b^2 + b)^{1/4})$, $\ka(\pp_1)=K(a^{1/2}, b^{1/2})$, $\ka(\pp_2) = K(a^{1/2})$, $\ka(\pp_3)=K$.
\end{enumerate}
\end{thm}


\begin{proof}
Since $F|K$ is non-hyperelliptic the singular prime $\pp$ has geometric singularity degrees $\de(\pp)=3$ and $\de(\pp_1)=1$ (see the beginning of Section~\ref{2024_04_05_18:25}).
Thus the function fields $F|K$ in the theorem are necessarily given as in Theorem~\ref{2024_02_18_00:35} or as in Theorem~\ref{2021_05_21_18:40}, \ref{2021_05_21_18:42} or~\ref{2021_05_21_18:43}.

(a) Conversely, a function field $F|K$ as in Theorem~\ref{2024_02_18_00:35}
is non-hyperelliptic if and only if $B_1\neq 0$ (see Proposition~\ref{2024_08_04_19:50}), in which case we can normalize $B_0=0$ by replacing $x$ with $x + B_0 B_1^{-1}$.
A quartic equation between the generators $\breve z = z \cdot (c_0 + c_1 x + x^2)^{-1}$ and $\breve y = y \cdot (c_0 + c_1 x + x^2)^{-1}$ of $F=K(\breve z,\breve y)$ is obtained by eliminating $x$ from the relations $\breve y^2 = c_1 B_1 x (A_2 + \breve z^2) + \breve z$ and $x^2 + c_1 x + c_0 = c_1^{-1}(A_2 + \breve z^2)^{-1}$, i.e.,
\[ B_1^{-2} c_1^{-2} \breve y^4 + c_0 \breve z^4 + B_1^{-1} \breve y^2 \breve z^2 + B_1^{-1}\breve z^3 + (c_1^{-1} + B_1^{-2} c_1^{-2}) \breve z^2 + B_1^{-1} A_2 \breve y^2 + B_1^{-1} A_2 \breve z + c_1^{-1} A_2 + c_0 A_2^2 = 0. \]
To get the desired equation we write $y$ and $z$ for $\breve y$ and $\breve z$ respectively, and we put $a:= A_2$, $b:= c_1^{-2} B_1^{-1}$, $c:= c_1 B_1$, $d:= B_1 c_0$, so that the mapping $(A_2,B_1,c_1,c_0) \mapsto (a,b,c,d)$ defines a bijection of the set $(K\setminus K^2)\times K^* \times  K^* \times K$ onto itself, thereby obtaining
\[ b y^4 + d z^4 + y^2 z^2 + z^3 + (b + b^2 c^3) z^2 + a y^2 + a z + a b^2 c^3 + a^2 d = 0, \]
where $a,b,c,d \in K$ are constants satisfying $a \notin K^2$ and $b,c\neq 0$.

(b)
Each function field $F|K$ as in Theorem~\ref{2021_05_21_18:40}~\ref{2021_05_21_18:42} is non-hyperelliptic by Proposition~\ref{2024_08_04_19:50}.
As $w= \breve w x$ and $y = \breve y x$, the third and the first equation mean that $z = x \breve y^2$ and $x = (a_2 + \breve w^2)^{-1}$. Thus the function field $F|K$ is generated by $\breve w$ and $\breve y$, and the second equation provides the quartic equation between the generators
\[ \breve y^4 = b_2 + b_0 a_2^2 + a_2 \breve w + \breve w^3 + b_0 \breve w^4. \]
Writing $y$ and $z$ for $\breve y$ and $\breve w$ respectively, and setting $a:=b_0$, $b:=a_2$ and $c:=b_2+b_0 a_2^2$, we obtain the representation of $F|K$ in the second item.

(c)
By Proposition~\ref{2024_08_04_19:50}, a function field $F|K$ as in Theorem~\ref{2021_05_21_18:40}~\ref{2021_05_21_18:43} is non-hyperelliptic if and only if $c_4\neq 0$, in which case we normalize $c_3=0$ by replacing $x$ with $x + c_4^{-1}c_3$.
As $z = w \breve z$ and $y = w \breve y$ the second and the third equations rewrite as follows $\breve z^2 = b_1+w^{-1}$ and $\breve y^2 = c_4 x w^{-1} + \breve z$. Eliminating $w$ and $x$ from the second and the third equation we conclude that $F=K(\breve z,\breve y)$, and entering into the first equation we obtain the quartic equation between the two generators
\[ a_2 \breve y^4 + c_4 \breve{z}^2 \breve y^2 + c_4^2 a_0 \breve{z}^4 + c_4 \breve{z}^3 + c_4 b_1 \breve y^2 + a_2 \breve{z}^2 + c_4 b_1 \breve{z} + c_4^2 (b_1^2 a_0 + 1) = 0. \]
Writing $y$ and $z$ for $\breve y$ and $\breve z$, and putting $a:= a_2^{-1} b_1^{-2} c_4^2$, $b= b_1$, $c:= a_2^{-1} c_4^2 (a_0 + b_1^{-2})$ and $d:= c_4 a_2^{-1}$ gives
\[ y^4 + d z^2 y^2 + (c+a) z^4 + d z^3 + bd\, y^2 + z^2 + bd\, z + b^2 c = 0, \]
where $a,b,c,d\in K$ satisfy $a,b\notin K^2$ and $d\neq 0$.

The point where the singular prime $\pp$ is centered is obtained in each case from the Jacobian criterion. The assertions on the residue fields of $\pp$ are due to Theorems~\ref{2024_02_18_00:35} and~\ref{2021_05_21_18:40}.
\end{proof}

Using Propositions~\ref{2024_03_25_19:55} and~\ref{2024_05_10_01:40}, together with Corollary~\ref{2024_08_04_20:25}, we obtain the isomorphism classes of the function fields in the above theorem.

\begin{prop} \label{2024_07_22_21:20} 
Let $F|K=K(z,y)|K$ and $F'|K=K(z',y')|K$ be two function fields as in Theorem~\ref{2024_05_27_00:10}.
\begin{enumerate}[\upshape (a)]
    \item 
    If $F|K$ and $F'|K$ are of type~\ref{2024_05_27_00:11}, then they are isomorphic if and only if there exist constants $\mu_2,\mu_3,\mu_4,\mu_5 \in K$ satisfying $(\mu_4,\mu_5)\neq (0,0)$ and
    \begin{align*}
        \eps^6 a' &= a + \ga^2, \qquad \eps^3 b' = b, \qquad c' = \eps c, \\
        d' &= \eps ( \mu_4 \mu_5 + \mu_3^2)^2 b + \eps^2 (\mu_4 \mu_5 + \mu_3^2) + \eps^2 ( \mu_5^2 b^2 c^3 + \eps d),
    \end{align*}
    where $\eps:= \mu_4^2 + \mu_5^2 a \neq 0$ and $\ga := \eps^{-1}(\mu_2^2 + \mu_3^2 a)$.
    The corresponding $K$-isomorphisms $F' \overset\sim\to F$ are given by
    \begin{align*}
        z' &\mapsto \frac 1 {\eps^3} \bigg( \ga + \frac{\mu_5 a + \mu_4 z}{\mu_4 + \mu_5 z} \bigg), \quad
        y' \mapsto \frac 1 {\eps^2} \bigg( \mu_2 + \frac{\mu_3 (\mu_5 a + \mu_4 z) + \eps y}{\mu_4 + \mu_5 z} \bigg).
    \end{align*}
    
    \item 
    If $F|K$ and $F'|K$ are of type~\ref{2024_05_27_00:12}, then they are isomorphic if and only if there exist constants $\mu_1,\mu_2,\mu_4,\mu_5 \in K$ satisfying $(\mu_4,\mu_5) \neq (0,0)$ and
    \begin{align*}
        a' &= \eps^2 a + \eps \mu_4 \mu_5 + \mu_2^4 + \mu_5^4 (c + a b^2), \qquad \eps^4 b' = b, \\
        \eps^6 c' &= c + \ga^2 + (\eps \mu_4 \mu_5 + \mu_2^4) \eps^{-2} b^2 + \eps^{-2} \mu_5^4 (c + ab^2) b^2,
    \end{align*}
    where $\eps := \mu_4^2 + \mu_5^2 b \neq 0$ and $\ga := \eps^{-1} (\mu_1^2 + \mu_2^2 b)$.
    The corresponding $K$-isomorphisms $F' \overset\sim\to F$ are given by
    \begin{align*}
        z' &\mapsto \frac 1 {\eps^2} \bigg( \frac{\mu_5 b + \mu_4 z}{\mu_4 + \mu_5 z}\bigg), \quad
        y' \mapsto \frac 1 {\eps^2} \bigg( \mu_1 + \frac{\mu_2 (\mu_5 b + \mu_4 z) + \eps y}{\mu_4 + \mu_5 z} \bigg).
    \end{align*}
    
    \item 
    If $F|K$ and $F'|K$ are of type~\ref{2024_05_27_00:13}, then they are isomorphic if and only if there exist constants $\mu_2,\mu_3,\mu_4,\mu_5 \in K$ satisfying $(\mu_4,\mu_5) \neq (0,0)$ and
    \begin{align*}
        a' &= \eps^2 \, \frac {ab^2} {(b + \ga^2)^2}, \qquad \eps^2 b' = b + \ga^2, \qquad d' = \eps d, \\
        c' &= \eps^2 (c + a) + (1 + \eps d) (\mu_3^2 + \mu_4 \mu_5) + a b^2 \Big( \mu_5^4 + \frac{\eps^2}{(b + \ga^2)^2} \Big),
    \end{align*}
    where $\eps := \mu_4^2 + \mu_5^2 b \neq 0$ and
    $\ga := \eps^{-1} (\mu_2^2 + \mu_3^2 b)$.
    The corresponding $K$-isomorphisms $F' \overset\sim\to F$ are given by
    \begin{align*}
        z' &\mapsto \frac 1 \eps \bigg( \ga + \frac{\mu_5 b + \mu_4 z}{\mu_4 + \mu_5 z}\bigg), \quad
        y' \mapsto \frac 1 \eps \bigg( \mu_2 + \frac{\mu_3 (\mu_5 b + \mu_4 z) + \eps y}{\mu_4 + \mu_5 z} \bigg).
    \end{align*}
\end{enumerate}
\end{prop}

\begin{cor}\label{2024_08_04_11:55}
Let $F|K$ be a function field as in Theorem~\ref{2024_05_27_00:10}.
Then the group $\mathrm{Aut}(F|K)$ of automorphisms of $F|K$ is trivial.
If $F|K$ is of type~\ref{2024_05_27_00:11} or~\ref{2024_05_27_00:13} then
\begin{enumerate}[\upshape (a)]
    \item $\iota = b c^3$,
    \setcounter{enumi}{2}
    \item $\iota = a b^2 d^2$,
\end{enumerate}
is an invariant of the function field $F|K$. 
\end{cor}

The corollary is also a consequence of Corollary~\ref{2024_03_25_11:50} and Proposition~\ref{2024_05_10_01:40}.
We conclude that there exist infinitely many isomorphism classes of function fields of type~\ref{2024_05_27_00:11} and~\ref{2024_05_27_00:13} (see also Remark~\ref{2024_08_25_19:50}).

Let $F|K$ be a non-hyperelliptic geometrically rational function field of genus $g=3$, whose only singular prime $\pp$ has the property that it is not a canonical divisor.
Then $F|K$ is the function field of a regular projective geometrically integral quartic curve $C$ defined over $\Spec(K)$, as described in the three items \ref{2024_05_27_00:11}, \ref{2024_05_27_00:12} and \ref{2024_05_27_00:13} in Theorem~\ref{2024_05_27_00:10}.
The curve $C|K$ is actually the regular complete model of the function field $F|K$; in other words, its closed points correspond bijectively to the primes of $F|K$, and its local rings are the local rings of the corresponding primes, which are discrete valuation rings, i.e., regular one-dimensional local rings.
The generic point is the only non-closed point, and its local ring is equal to the function field $F$.

The extended curve $C_{\ov K} = C \otimes_K \ov K$ is a plane quartic rational curve in $\PP^2 (\ov K)$, with a unique singular point at $(1:a^{1/4}:a^{1/2})$, $(1:(a b^2 + c)^{1/4} : b^{1/2})$ and $(1:(a b^2 + b)^{1/4}:b^{1/2})$ respectively.
This is the point at which the singular prime $\pp$ is centered.
Its local ring $\OO_\pp \otimes_K \ov K$ has $\delta$-invariant 3, and 
moreover its tangent line intersects the quartic curve uniquely at this point.
If item~\ref{2024_05_27_00:12} or item~\ref{2024_05_27_00:13} occurs then this tangent line has multiplicity 2, whereas for item~\ref{2024_05_27_00:11} the multiplicity may be 2 (if $b\neq c^{-3}$) or 3 (if $b = c^{-3}$).

The quartic curve $C_{\ov K}$ is \emph{strange}, in the sense that all its tangent lines meet at the common point intersection point $(0:1:0)$.
In items~\ref{2024_05_27_00:11} and~\ref{2024_05_27_00:13} the tangent lines at the non-singular points of $C_{\ov K}$ are all bitangents, so in particular the quartic curve $C_{\ov K}$ has no inflection points.
In item~\ref{2024_05_27_00:12} the two points of tangency coincide, and so each non-singular point is a non-ordinary inflection point.

\medskip 

Combined with \cite[Theorem~1.1]{HiSt23}, Theorem~\ref{2024_05_27_00:10} yields a complete classification of all regular geometrically rational plane projective quartic curves in characteristic $p=2$, as stated in Theorem~\ref{2024_06_01_23:20}.

\begin{rem}
Let $C|K$ be a curve defined as in Theorem~\ref{2024_06_01_23:20}, item \ref{2024_06_01_23:21} or~\ref{2024_06_01_23:22}, and let $F|K=K(C)|K$ be its function field. 
It can be shown by similar considerations that a basis of holomorphic differentials and the pseudocanonical field are given by
\begin{enumerate}[(i)]
   \item $d \breve y$, $\breve x \, d \breve y$, $\breve y \, d \breve y$ and $E=K(\breve x)=F_2$, where $\breve x=x/z\in F$ and $\breve y = y/z \in F$;
   \item $d\breve y$, $\breve z \, d\breve y$, $\breve y \, d \breve y$ and $E=K(\breve z)\neq F_2$, where $\breve y=y/x\in F$ and $\breve z = z/x \in F$.
\end{enumerate}
Thus the five families in Theorem~\ref{2024_06_01_23:20} can be distinguished by intrinsic properties, as documented in Table~\ref{2025_04_03_02:45}.
\end{rem}

\section{Universal fibrations by plane projective rational quartic curves}

We now look at the fibrations that can be constructed with the three families of curves in Theorem~\ref{2024_05_27_00:10}, or equivalently, with the last three families in Theorem~\ref{2024_06_01_23:20}.
To this end we fix an algebraically closed ground field $k$ of characteristic 2.
We consider the three integral hypersurfaces
\[ Z_3 \su \PP^2 \times \AAA^4, \quad Z_4 \su \PP^2 \times \AAA^3, \quad Z_5 \su \PP^2 \times \AAA^4 \]
whose points $((x:y:z),(a,b,c,d))$, $((x:y:z),(a,b,c))$, $((x:y:z),(a,b,c,d))$ satisfy the relations
\begin{gather*}
    b y^4 + d z^4 + y^2 z^2 + x z^3 + (b + b^2 c^3) x^2 z^2 + a x^2 y^2 + a x^3 z + (a b^2 c^3 + a^2 d) x^4 = 0, \\
    y^4 + a z^4 + x z^3 + b x^3 z + c x^4 = 0, \\
    y^4 + d z^2 y^2 + (c+a) z^4 + d x z^3 + bd\, x^2 y^2 + x^2 z^2 + bd\, x^3 z + b^2 c x^4 = 0,
\end{gather*}
respectively.
The projection morphisms
\[ \pi_3:Z_3\to \AAA^4, \quad \pi_4: Z_4\to \AAA^3, \quad 
\pi_5:Z_5 \to \AAA^4  \]
yield families of plane projective rational quartic curves, whose generic fibres are curves defined over the function fields of the bases, given as in items~\ref{2024_05_27_00:11}, \ref{2024_05_27_00:12} and~\ref{2024_05_27_00:13} in Theorem~\ref{2024_05_27_00:10}, respectively.

\begin{thm}\label{2024_05_29_21:05}
Let $\phi : T \to B$ be a proper dominant morphism of irreducible smooth
algebraic varieties whose generic fibre is a non-hyperelliptic geometrically rational curve $C$ over $k(B)$ of genus $3$, that admits a non-smooth point $\pp$ that is not a canonical divisor on $C$. 
Then the fibration $\phi : T \to B$ is, up to birational equivalence, a dominant base extension of one of the fibrations $\pi_3$, $\pi_4$ or $\pi_5$.
\end{thm}

The theorem states that the three fibrations $\pi_i$ are universal in the sense that any fibration whose generic fibre satisfies the hypotheses of  Theorem~\ref{2024_05_27_00:10} is obtained from one of them by a base extension.
Its proof is analogous to that of \cite[Theorem~5.1]{HiSt23}, and so we omit it.

We next describe the fibres of the fibrations $\pi_i$. 
Generically, they behave in the same manner as the geometric generic fibres $C_{\ov K}$ discussed at the end of the preceding section. But there are special fibres that behave differently.

We start with the fibration $\pi_3:Z_3\to \AAA^4$.
If $b=0$ then the fibre $\pi_3^{-1}(P)$ over the point $P=(a,b,c)$ is the union of a smooth quadric and a double line, hence it is reducible and non-reduced.
If $b\neq0$ then we are in the generic case, where the fibre behaves in the same way as the geometric generic fibre $C_{\ov K}$ described in the previous section. 
Precisely, $\pi_3^{-1}(P)$ is a plane rational quartic curve with a unique singular point of multiplicity 2 (if $\iota := b c^3 \neq 1$) or 3 (if $\iota = b c^3 = 1$), 
whose unique tangent line meets the curve only at that point; moreover, each of the remaining tangent lines is a bitangent, and all of them pass through a common intersection point, i.e., the curve is strange.

Next we consider the fibration $\pi_4:Z_4\to \AAA^3$.
If $b\neq 0$ then we are in the generic case, i.e., the fibre $\pi_4^{-1}(P)$ over the point $P=(a,b,c)$ is a plane rational quartic curve with a unique singular point of multiplicity 2, and with all its tangent lines meeting in a unique common point, i.e., the curve is strange; in addition, every tangent line intersects the curve at a unique point, hence every non-singular point is a non-ordinary inflection point.
In the opposite case $b=0$ the fibre has the same properties as the fibres in the generic case, with only one difference: the singular point has multiplicity 3.

We next we consider the fibration $\pi_5:Z_5\to \AAA^4$.
If $d=0$ then the fibre $\pi_5^{-1}(P)$ over the point $P=(a,b,c,d)$ is a double smooth quadric and is therefore non-reduced.
In the opposite case $d\neq0$, the generic case occurs if and only if 
$\iota := ab^2 d^{2}\neq 1$.
In other words, if $\iota \neq 1$ then the fibre is a plane rational quartic curve satisfying the following conditions: it has a unique singular point, which has a unique tangent line of multiplicity 2; all its tangent lines intersect in a unique common point; the tangent lines at the non-singular points are bitangents.
If $\iota = 1$, then the fibre has the same properties as in the generic case, the only difference being that the multiplicity of the singular point becomes 3.

The total space $Z_4$ of the fibration $\pi_4:Z_4\to \AAA^4$ is smooth. 
The total spaces $Z_3$ and $Z_5$ of $\pi_3:Z_3\to \AAA^3$ and $\pi_5:Z_5\to \AAA^4$, on the other hand, are not smooth, but they become smooth after restricting the bases $\AAA^3$ and $\AAA^4$ to the dense open subsets $\{ bc \neq 0\}$ and $\{ b\neq 0 \}$ respectively.

Combining Theorem~\ref{2024_05_29_21:05} with Theorems 5.1 and 5.2 in \cite{HiSt23} we obtain five universal fibrations from which every fibration by plane rational quartic curves arises as a base extension.


We end this section by showing that the total spaces $Z_i$ are uniruled.
This will be a consequence of a more general result, valid in any characteristic $p>0$.

\begin{prop}\label{2024_09_26_19:35}
Let $k$ be an algebraically closed ground field of characteristic $p>0$.
Let $T\to B$ be a dominant morphism of integral varieties whose generic fibre $C=T_\eta$ is a regular geometrically integral curve over $K=k(B)$.
Assume that for some $n$ the Frobenius pullback $C^{(p^n)}|K$ is rational.
Then the total space $T$ is (inseparably) uniruled.
If in addition the base $B$ is rational then $T$ is (inseparably) unirational.
\end{prop}

By specializing $n=3$ we obtain the uniruledness of the varieties $Z_i$.
The cases $n=1$ and $n=2$ imply the well-known fact that quasi-elliptic surfaces are uniruled (see e.g. \cite[Theorem~9.4]{Lied13} or \cite[Corollary~4.1.16]{CDL23}).

\begin{rem}
According to \cite[Lemma~1.2]{Sc09} or \cite[Corollary~2.7]{HiSt22}, the assumption that $C^{(p^n)}|K$ is rational for some $n$ is equivalent to the condition that the curve $C$ is geometrically rational, i.e., the geometric generic fibre $C_{\ov K}=C\otimes_K \ov K$ is a rational curve over $\ov K$, i.e., most of the fibres of $T\to B$ are rational curves.
\end{rem}

\begin{proof}

Note first that the generic fibre $C$ of the fibration $T\to B$ has function field
\[ F := K(C) = k(T). \]
If $n=0$ then by assumption the generic fibre $C|K$ is rational, say $F=K(t)$ for some $t$ transcendental over $K$, i.e., $k(T)=k(B)(t)$, and so the total space $T$ is  birationally equivalent to $B\times \PP^1$, i.e., it is ruled over $B$. 

Let us look at the case $n>0$.
Let $B'$ be the integral $\Spec(k)$-scheme whose underlying topological space is equal to $B$ and whose spaces of sections are the sub-$k$-algebras of $\ov K$ defined by
$ \OO_{B'}(U) := \OO_B(U)^{p^{-n}}$, with $U\su B$ open.
This scheme has function field $k(B') = K' := K^{p^{-n}}$, and it is isomorphic, as an abstract scheme, to $B$ via the $n$th absolute Frobenius morphism $B \to B'$ that raises sections to their $p^n$-powers.
However, $B$ and $B'$ are not necessarily isomorphic as $\Spec(k)$-schemes.
In fact $B'$ is, up to isomorphism, the $\Spec(k)$-scheme
$ B \to \Spec(k) \to \Spec(k) $
whose structure morphism has been modified by composing with the $n$th absolute Frobenius morphism $F_k^n : \Spec(k) \to \Spec(k)$.

The inclusion homomorphisms $\OO_B(U)\to \OO_{B'}(U)$ define a purely inseparable $\Spec(k)$-morphism
$ B' \to B$
of degree $[K':K]=p^{n \dim(B)}$.
In turn we get an extended fibration $T':= T \times_B B'\to B'$, whose total space $T'$ has function field 
\[ F' := k(T') = F K' = F K^{p^{-n}}. \]
In particular the morphism $T' \to T$ is purely inseparable of degree $[F':F]=[K':K]=p^{n \dim (B)}$.
Moreover the function field $F'|K'$ of $T'\to B'$ is isomorphic to the function field $F^{p^n} {\cdot} K|K$ of the Frobenius pullback $C^{(p^n)}|K$, via the $n$th absolute Frobenius homomorphism.
It follows that the function field $F'|K'$ is rational by assumption, and as in the case $n=0$ we conclude that $T'$ is birationally equivalent to $B'\times \PP^1$.
Thus the composition 
$ B'\times \PP^1 \ttto T' \to T $ is a dominant purely inseparable rational map of degree $p^{n \dim(B)}$,
thereby proving that $T$ is (inseparably) uniruled. 

To complete the proof it remains to note that if $B$ is rational then so is $B'$.
\end{proof}

For a similar proof in terms of classical algebraic geometry we refer to \cite[Section~5]{HiSt23}.
Note that since $T$ is \emph{inseparably} uniruled, it does not necessarily have negative Kodaira dimension (see \cite[p.\,265]{Lied13}).

Since the normalized Frobenius pullbacks $C_1|K$ of the curves $C|K$ in the preceding section are quasi-elliptic, the proof of the proposition yields the following consequence for surfaces (compare \cite[Corollary~4.1.16]{CDL23}).
As before, we assume that the algebraically closed ground field $k$ has characteristic $p=2$.

\begin{cor}\label{2025_10_10_16:30}
Every smooth surface admitting a fibration by rational quartic curves admits a purely inseparable cover of degree $p=2$ by a quasi-elliptic surface.
\end{cor}

\section{A pencil of rational quartics in characteristic two}
\label{2024_06_02_00:30}

In this section we discuss the pencil of rational quartics obtained from Theorem~\ref{2024_05_27_00:10}~\ref{2024_05_27_00:12}, by specializing $a=c=0$.

\medskip

Let $k$ be an algebraically closed ground field of characteristic two.
We consider the integral projective algebraic surface
\[ S\su \PP^2 \times \PP^1 \]
of the pairs $((x:y:z),(t_0:t_1))$ that satisfy the bihomogeneous equation
\[ t_0 (y^4 + x z^3) + t_1 x^3 z = 0. \]
This defines a pencil of plane projective quartic curves, whose base points (i.e., the common points of its members) are equal to $(1:0:0)$ and $(0:0:1)$.
For each point of the form $(1:c)$ in $\PP^1$, the corresponding member is the plane projective integral curve cut out by the equation
\[ y^4 + xz^3+ c x^3 z = 0. \]
The only singular point $(1:0:c^{1/2})$ of the curve is unibranch of multiplicity $2$ (if $c \neq 0$) or $3$ (if $c = 0$).
The tangent line at the singular point as well as the tangent line at each non-singular point does not intersect the curve at any other point.
Thus the non-singular points of the quartic curve are non-ordinary inflection points.
Moreover the curve is strange, because its tangent lines pass through the common point $(0:1:0)$.
The member of the pencil corresponding to the point $(0:1)$, on the other hand, degenerates to the non-reduced reducible curve $V(x z^3)$ consisting of the line $V(x)$ and the triple line $V(z^3)$.

By the Jacobian criterion, the surface $S\subset \PP^2 \times \PP^1$ has exactly two singular points, namely $P:=((1:0:0),(1:0))$ and $Q:=((0:0:1),(0:1))$. The fibres of the second projection morphism
\[ S\tto \PP^1 \]
are up to isomorphism the members of the above pencil.
The first projection
\[ S\tto \PP^2 \]
is a birational morphism whose inverse
\[ \PP^2 \ttto S, \quad (x:y:z) \mapsto \big( (x:y:z),(x^3 z : y^4 + x z^3) \big) \]
is not defined only at the base points $(1:0:0)$ and $(0:0:1)$.
More precisely, the morphism $S \to \PP^2$ contracts the horizontal lines $(1:0:0) \times \PP^1$ and $(0:0:1) \times \PP^1$ to the base points $(1:0:0)$ and $(0:0:1)$, and restricts to an isomorphism
\[ S \setminus ( (1:0:0) \times \PP^1 \cup (0:0:1) \times \PP^1 )  \overset \sim \tto \PP^2 \setminus \{ (1:0:0),(0:0:1) \}. \]
The inverse map $\PP^2 \ttto S \subset \PP^2 \times \PP^1$ is equal to $(\textrm{id},\tau)$, where $\tau$ is the rational map $\PP^2 \ttto S \to \PP^1$, i.e.,
\[ \tau: \PP^2 \ttto \PP^1, \quad (x:y:z) \mapsto (x^3 z : y^4 + x z^3), \]
which is also undefined only at $(1:0:0)$ and $(0:0:1)$.

To resolve the singularities of the rational surface $S\subset \PP^2 \times \PP^1$ we resolve the indeterminacies of the rational map $\tau:\PP^2 \ttto \PP^1$. 
This provides us with a commutative diagram
\[ 
\begin{tikzcd}
    \wt S \ar[d,"\lam"]\ar[dr,"\wt \tau", bend left = 20] \\
    \PP^2 \ar[r,dashed,"\tau"] & \PP^1 
\end{tikzcd} 
\]
where $\wt S$ is a smooth integral projective surface and where $\lam: \wt S \to \PP^2$ and $\wt \tau: \wt S \to \PP^1$ are morphisms \cite[p.\,263, Theorem~4.8]{Shaf13a}.
This pair of morphisms defines the desingularization morphism
\[ (\lam,\wt \tau): \wt S \tto S \subset \PP^2 \times \PP^1. \]
For a more standard treatment of the singularities of a similar fibration we refer to \cite[Section 6]{HiSt23}.
As the intersection multiplicities of the two members $V(y^4 + xz^3)$ and $V(x^3 z)$ at the base points $(1:0:0)$ and $(0:0:1)$ are equal to $4$ and $12$, the morphism $\lam: \wt S \to \PP^2$ is obtained by a chain of 4 blowups over $(1:0:0)$ and 12 blowups over $(0:0:1)$, as follows from \cite[p.\,262, Corollary~4.3]{Shaf13a}.
More precisely, computation shows that the exceptional curves (i.e., the birational transforms of the exceptional lines), say $E_i$ ($i=1,\dots,4$) and $F_i$ ($i=1,\dots,12$), intersect transversely according to the Dynkin diagrams $A_4$ and $A_{12}$, i.e.,
\[
\centering
\begin{tikzpicture}[line cap=round,line join=round,x=0.8cm,y=0.8cm]
\draw [line width=1.2pt] (-7.,0.)-- (-4.,0.);
\begin{scriptsize}
\draw [fill=black] (-7,0) circle (2pt);
\draw[color=black] (-7,0.5) node {$E_{1}$};    
\draw [fill=black] (-6,0) circle (2pt);
\draw[color=black] (-6,0.5) node {$E_{2}$};
\draw [fill=black] (-5,0) circle (2pt);
\draw[color=black] (-5,0.5) node {$E_3$};
\draw [fill=black] (-4,0) circle (2pt);
\draw[color=black] (-4,0.5) node {$E_4$};
\end{scriptsize}
\end{tikzpicture}
\]
and
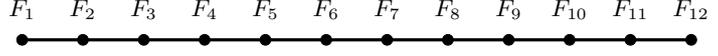
\[
\centering
\begin{tikzpicture}[line cap=round,line join=round,x=0.8cm,y=0.8cm]
\draw [line width=1.2pt] (-7.,0.)-- (4.,0.);
\begin{scriptsize}
\draw [fill=black] (-7,0) circle (2pt);
\draw[color=black] (-7,0.5) node {$F_{1}$};    
\draw [fill=black] (-6,0) circle (2pt);
\draw[color=black] (-6,0.5) node {$F_{2}$};
\draw [fill=black] (-5,0) circle (2pt);
\draw[color=black] (-5,0.5) node {$F_3$};
\draw [fill=black] (-4,0) circle (2pt);
\draw[color=black] (-4,0.5) node {$F_4$};
\draw [fill=black] (-3,0) circle (2pt);
\draw[color=black] (-3,0.5) node {$F_5$};
\draw [fill=black] (-2,0) circle (2pt);
\draw[color=black] (-2,0.5) node {$F_6$};
\draw [fill=black] (-1,0) circle (2pt);
\draw[color=black] (-1,0.5) node {$F_7$};
\draw [fill=black] (0,0) circle (2pt);
\draw[color=black] (0,0.5) node {$F_8$};
\draw [fill=black] (1,0) circle (2pt);
\draw[color=black] (1,0.5) node {$F_9$};
\draw [fill=black] (2,0) circle (2pt);
\draw[color=black] (2,0.5) node {$F_{10}$};
\draw [fill=black] (3,0) circle (2pt);
\draw[color=black] (3,0.5) node {$F_{11}$};
\draw [fill=black] (4,0) circle (2pt);
\draw[color=black] (4,0.5) node {$F_{12}$};
\end{scriptsize}
\end{tikzpicture}.
\]
Moreover $E:=E_4$ and $F:= F_{12}$ are horizontal curves, namely the birational transforms of the horizontal lines $(1:0:0) \times \PP^1$ and $(0:0:1) \times \PP^1$ under the desingularization morphism $(\lam,\wt \tau): \wt S \to S \subset \PP^2 \times \PP^1$.
The self-intersection numbers are equal to
\begin{align*}
    E_i \cdot E_i &= -2 \quad (i=1,2,3), \quad E \cdot E = -1, \\
    F_i \cdot F_i &= -2 \quad (i = 1,\dots,11), \quad F \cdot F = -1.
\end{align*}
The fibres of the morphism $\wt \tau: \wt S \to \PP^1$ over $(1:0)$ and $(0:1)$ contain the exceptional curves $E_i$ except $E_4 = E$, and $F_i$ except $F_{12}=F$, respectively.
Hence the singular points $P$ and $Q$ on the surface $S$ are rational double points of type $A_3$ and $A_{11}$, respectively.

We denote by $W$, $X$ and $Z$ the birational transforms of the plane projective curves $V(y^4 + xz^3)$, $V(x)$ and $V(z)$ under the morphism $\lam: \wt S \to \PP^2$.
Note that $W$, $X$ and $Z$ are smooth rational curves; indeed the singularity of $V(y^4 + xz^3)$ is resolved by the first blowup over $(1:0:0)$.

The computations also show that the fibre of $\wt \tau$ over the point $(1:0)$ is the Weil divisor
\begin{equation}\label{2024_07_22_13:30}
    \wt \tau^* (1:0) = W + 2 E_1 + 2 E_2 + E_3,
\end{equation}
where the components intersect according to the configuration
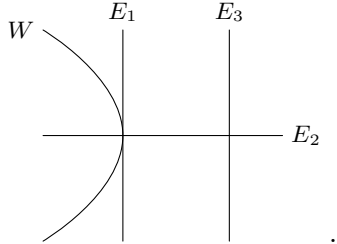
\[
\begin{scriptsize}
\begin{tikzpicture}[line cap=round,line join=round,x=0.7cm,y=0.7cm]
\draw plot [smooth,tension=1] coordinates {(-1.5,-2) (0,0) (-1.5,2)}; \draw (-1.5,2) node[left] {$W$};
\draw (-1.5,0) -- (3,0) node[right] {$E_2$};
\draw (0,-2) -- (0,2) node[above] {$E_1$};
\draw (2,-2) -- (2,2) node[above] {$E_3$};

\end{tikzpicture}
\end{scriptsize}.
\]
The intersection matrix of $W$, $E_1$, $E_2$, $E_3$ is equal to
\[ \left( \begin{array}{cccc} 
-6 & 2 & 1 & 0  \\
2 & -2 & 1 & 0 \\
1 & 1 & -2 & 1 \\
0 & 0 & 1 & -2 \\
\end{array} \right) \]
where the self-intersection numbers can be obtained from the property that a fibre meets each of its components with intersection number zero.

The fibre of $\wt \tau$ over the point $(0:1)$ is the Weil divisor
\begin{equation}\label{2024_07_22_13:35}
\begin{aligned}
    \wt \tau^*(0:1) &= 3X + Z + 2 F_1 + 4 F_2 + 6 F_3 + 8 F_4 + 7 F_5 + 6 F_6  \\
    &\qquad + 5 F_7 + 4 F_8 + 3 F_9 + 2 F_{10} + F_{11}
\end{aligned}
\end{equation}
whose components intersect transversely according to the Coxeter-Dynkin diagram
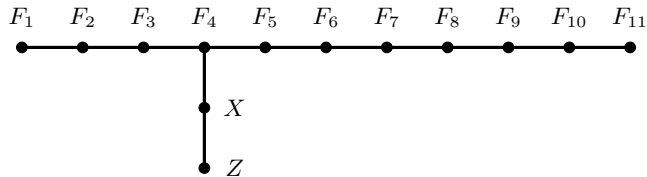
\[
\begin{tikzpicture}[line cap=round,line join=round,x=0.8cm,y=0.8cm]
\draw [line width=1.2pt] (-7.,0.)-- (3.,0.);
\draw [line width=1.2pt] (-4.,0.)-- (-4.,-2.);
\begin{scriptsize}
\draw [fill=black] (-7,0) circle (2pt);
\draw[color=black] (-7,0.5) node {$F_{1}$};    
\draw [fill=black] (-6,0) circle (2pt);
\draw[color=black] (-6,0.5) node {$F_{2}$};
\draw [fill=black] (-5,0) circle (2pt);
\draw[color=black] (-5,0.5) node {$F_3$};
\draw [fill=black] (-4,0) circle (2pt);
\draw[color=black] (-4,0.5) node {$F_4$};
\draw [fill=black] (-4,-1.) circle (2pt);
\draw[color=black] (-4.5 + 1,-1) node {$X$};
\draw [fill=black] (-4.,-2.) circle (2pt);
\draw[color=black] (-4.5 + 1,-2) node {$Z$};
\draw [fill=black] (-3,0) circle (2pt);
\draw[color=black] (-3,0.5) node {$F_5$};
\draw [fill=black] (-2.,0.) circle (2pt);
\draw[color=black] (-2.,.5) node {$F_6$};
\draw [fill=black] (-1.,0.) circle (2pt);
\draw[color=black] (-1.,.5) node {$F_7$};
\draw [fill=black] (0.,0.) circle (2pt);
\draw[color=black] (0.,.5) node {$F_8$};
\draw [fill=black] (1.,0.) circle (2pt);
\draw[color=black] (1.,0.5) node {$F_9$};
\draw [fill=black] (2.,0.) circle (2pt);
\draw[color=black] (2.,.5) node {$F_{10}$};
\draw [fill=black] (3.,0.) circle (2pt);
\draw[color=black] (3.,0.5) node {$F_{11}$};
\end{scriptsize}
\end{tikzpicture}
\]
Moreover $X\cdot X = Z \cdot Z = -3$.
Note also that $W \cdot F = Z \cdot E = 1$.

To summarize we put the curves appearing in the discussion of the bad fibres $\wt \tau^*(1:0)$ and $\wt \tau^*(0:1)$ into a unique configuration in Figure~\ref{2024_07_23_20:15}.
\begin{figure}[h]

\centering
\begin{tikzpicture}[line cap=round,line join=round,x=0.7cm,y=0.7cm]

\begin{scriptsize}


\draw (-6,0 - .6) -- (-6,2 + .6) ;  \draw[color=black] (-6,2 + .6) node[above] {$F_{11}$}; 
\draw (-4,0 - .6) -- (-4,2 + .6) ;  \draw[color=black] (-4,2 + .6) node[above] {$F_9$}; 
\draw (-2,0 - .6) -- (-2,2 + .6) ;  \draw[color=black] (-2,2 + .6) node[above] {$F_7$}; 
\draw (0,0 - .6) -- (0,2 + .6) ;  \draw[color=black] (0,2 + .6) node[above] {$F_5$}; 
\draw (2,0 - .6) -- (2,2 + .6) ;  \draw[color=black] (2,2 + .6) node[above] {$F_3$}; 
\draw (4,0 - .6) -- (4,2 + .6) ;  \draw[color=black] (4,2 + .6) node[above] {$F_1$}; 

\draw (-6- .6,2) -- (-4 + .6,2) ; \draw[color=black] (-5,2 + .0) node[above] {$F_{10}$}; 
\draw (-4- .6,0) -- (-2 + .6,0) ;  \draw[color=black] (-3,0 + .0) node[above] {$F_8$}; 
\draw (-2- .6,2) -- (0 + .6,2) ;  \draw[color=black] (-1,2 + .0) node[above] {$F_6$}; 
\draw (0- .6,0) -- (2 + .6,0) ;  \draw[color=black] (2 + .6,0) node[right] {$F_4$};
\draw (2- .6,2) -- (4 + .6,2) ;  \draw[color=black] (3,2 + .0) node[above] {$F_2$}; 

\draw (1,0.5) -- (2.5,-2); \draw[color=black] (1,0.5) node[above] {$X$};
\draw (1.5,-1.5) -- (6.1,-1.5); \draw[color=black] (6.1,-1.5) node[right] {$Z$}; 


\draw (-5 +3,-5 - .6) -- (-5 +3,-3 + .6) ; \draw[color=black] (-5 +3,-5 - .6) node[below] {$E_1$}; 
\draw (-2 +3,-5 - .6) -- (-2 +3,-3 + .6) ; \draw[color=black] (-2 +3,-5 - .6) node[below] {$E_3$}; 
\draw (-5- .6 +3,-4) -- (-2 + .6 +3,-4) ;  \draw[color=black] (-5 +3 + 1.5,-4) node[above] {$E_2$};
\draw plot [smooth,tension=1] coordinates {(-2 -5 +3,-1.5 -4) (0-5 +3,0-4) (-2 -5 +3,1.5 -4)}; \draw (-2 -5 +3,-1.5 -4) node[left] {$W$};

\draw (-6 - .6/2,0 + .6) -- (-3 - .2/2,-6 + .2); \draw[color=black] (-4.5,-3) node[left] {$F$}; 
\draw (1 - .6,-5 - .6*4/5) -- (6,-1); \draw[color=black] (3.5,-3) node[below] {$E$};

\end{scriptsize}
\end{tikzpicture}
\caption{Configuration of curves on $\wt S$}\label{2024_07_23_20:15}
\end{figure}
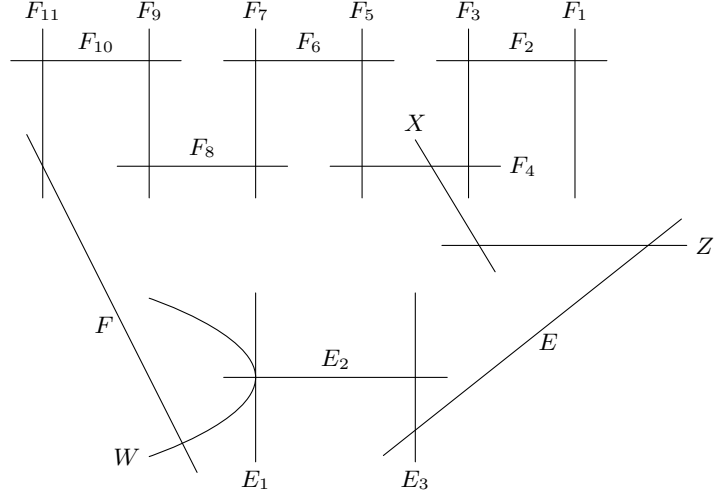

The remaining fibres of $\wt \tau$ are integral curves of self-intersection number zero.
As the fibres of $\wt \tau$ do not contain curves of self-intersection $-1$, that is, curves that are contractible according to Castelnuovo's contractibility criterion, we deduce that the fibration $\wt \tau: \wt S \to \PP^1$ is a relatively minimal model and hence by a theorem of Lichtenbaum--Shafarevich (see \cite[Theorem~4.4]{Lic68}, \cite[p.\,155]{Sha66}, or \cite[p.\,422]{Liu02}) the (unique) minimal model of the function field $F|K = k(S)|k(\PP^1)$.
However, as the two horizontal curves $E=E_4$ and $F=F_{12}$ are contractible, the smooth surface $\wt S$ is not relatively minimal over $\Spec(k)$.

Since the Frobenius pullback $F_1|K$ of the function field $F|K = k(\wt S)|k(\PP^1)$ is quasi-elliptic, it follows that the fibration $\wt \tau: \wt S \to \PP^1$ is an inseparable covering of degree two of a quasi-elliptic fibration.
In the remaining of this section we describe this covering.

Let $S' \su \PP^2 \times \PP^1$ be the integral projective algebraic surface of the pairs $((u:v:w),(t_0:t_1))$ satisfying the bihomogeneous equation
\[ t_0 (u v^2 + w^3) + t_1 u^2 w = 0. \]
This defines a pencil of plane projective cubic curves, whose base points are equal to $(1:0:0)$ and $(0:1:0)$.
The members of this pencil are up to isomorphisms the fibres of the second projection morphism $S\to \PP^1$.
These fibres are cuspidal cubic curves, except the fibre over the point $(0:1)$, which consists of the line $V(w)$ and the double line $V(u^2)$.

The surface $S' \su \PP^2 \times \PP^1$ has precisely two singularities at the points $P' = ((1:0:0),(1:0))$ and $Q'=((0:1:0),(0:1))$.
The first projection $S' \to \PP^2$ is a birational morphism, whose inverse is the rational map
\[ (\mathrm{id},\tau'):\PP^2 \ttto S' \subset \PP^2 \times \PP^1 \]
where $\tau':\PP^2 \ttto \PP^1$, $(u:v:w) \mapsto (u^2 w: uv^2 + w^3)$, is undefined only at the base points $(1:0:0)$ and $(0:1:0)$.
Resolving the indeterminacy of $\tau'$ we obtain the commutative diagram
\[
\begin{tikzcd}
    \wt S' \ar[d,"\lam'"] \ar[dr,"\wt \tau'", bend left=20] \\
    \PP^2 \ar[r,dashed,"\tau'"] & \PP^1
\end{tikzcd}
\]
where $\wt S'$ is a smooth integral projective surface and where $\lam'$ and $\wt \tau'$ are morphisms, which provide a desingularization morphism
\[ (\lam',\wt \tau') : \wt S' \tto S' \subset \PP^2 \times \PP^1 \]
of the variety $S'$.
The morphism $\lam':\wt S' \to \PP^2$ is obtained by a chain of 2 blowups over $(1:0:0)$ and $7$ blowups over $(0:1:0)$.
The corresponding exceptional curves $E_i'$ ($i=1,2$) and $F_i'$ ($i=1,\dots,7$) on $\wt S'$ intersect according to the Dynkin diagrams $A_2$ and $A_7$, respectively.
Moreover, $E':= E_2'$ and $F':= F_7'$ are horizontal curves of self-intersection number $-1$, namely the birational transforms of the horizontal lines $(1:0:0)\times \PP^1$ and $(0:1:0)\times \PP^1$ under the desingularization morphism $\wt S' \to S'$.
The singular points $P'$ and $Q'$ on the surface $S'$ are rational double points of type $A_1$ and $A_6$, respectively.

We denote by $W'$, $X'$ and $Z'$ the birational transforms of the plane projective curves $V(uv^2 + w^3)$, $V(u)$ and $V(w)$.
The fibres of the morphism $\wt \tau' : \wt S' \to \PP^1$ over the points $(1:0)$ and $(0:1)$ are of type $\wt A_1^*$ and $\wt E_7$ respectively (see \cite[Theorem~4.1.4 and Corollary~4.3.22]{CDL23}).
More precisely, these fibres are given by
\[ \wt \tau'^* (1:0) = W' + E_1' \]
where $W'$ and $E_1'$ meet with intersection number $2$ at only one point, and
\[ \wt \tau'^*(0:1) = 2 X' + Z' + 2 F'_1 + 3 F'_2 + 4 F'_3 + 3 F'_4 + 2 F'_5 + F'_6  \]
where the components intersect transversely according to the diagram
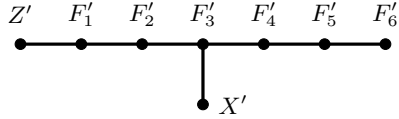
\[
\begin{tikzpicture}[line cap=round,line join=round,x=0.8cm,y=0.8cm]
\draw [line width=1.2pt] (-7.,0.)-- (-1.,0.);
\draw [line width=1.2pt] (-4.,0.)-- (-4.,-1.);
\begin{scriptsize}
\draw [fill=black] (-7,0) circle (2pt);
\draw[color=black] (-7,0.5) node {$Z'$};    
\draw [fill=black] (-6,0) circle (2pt);
\draw[color=black] (-6,0.5) node {$F_1'$};
\draw [fill=black] (-5,0) circle (2pt);
\draw[color=black] (-5,0.5) node {$F_2'$};
\draw [fill=black] (-4,0) circle (2pt);
\draw[color=black] (-4,0.5) node {$F_3'$};
\draw [fill=black] (-4,-1.) circle (2pt);
\draw[color=black] (-4.5 + 1,-1) node {$X'$};
\draw [fill=black] (-3,0) circle (2pt);
\draw[color=black] (-3,0.5) node {$F_4'$};
\draw [fill=black] (-2.,0.) circle (2pt);
\draw[color=black] (-2.,.5) node {$F_5'$};
\draw [fill=black] (-1.,0.) circle (2pt);
\draw[color=black] (-1.,.5) node {$F_6'$};
\end{scriptsize}
\end{tikzpicture}
\]
The self-intersection numbers of the curves $E_1',F_1',\dots,F_5',W',X',Z'$ are equal to $-2$.
The remaining fibres of $\wt \tau'$ are integral curves of self-intersection number zero.
Thus the fibration $\wt S' \to \PP^1$ is the minimal model of the Frobenius pullback $F_1|K$ of the function field $F|K$.

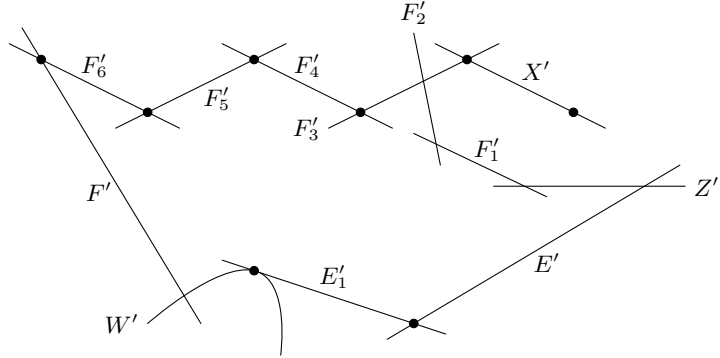
\begin{figure}[h]

\centering
\begin{tikzpicture}[line cap=round,line join=round,x=0.7cm,y=0.7cm]

\begin{scriptsize}

\draw (-8 -.6 +2,-4 + .6/2 -6) -- (-6 + .6 +2,-5 - .6/2 -6) ;  \draw[color=black] (-5,-10.5) node[above] {$F_6'$};
\draw (-4 -.6 +2,-4 + .6/2 -6) -- (-2 + .6 +2,-5 - .6/2 -6) ;  \draw[color=black] (-1,-10.5) node[above] {$F_4'$};
\draw (0 -.6 +2,-4 + .6/2 -6) -- (2 + .6 +2,-5 - .6/2 -6) ;  \draw[color=black] (3 + .3,-10.5 - .3/2 + .1) node[above] {$X'$};

\draw (-4 + .6 +2,-4 + .6/2 -6) -- (-6 - .6 +2,-5 - .6/2 -6) ;  \draw[color=black] (-3 + .3,-10.5 + .3/2) node[below] {$F_5'$}; 
\draw (0 + .6 +2,-4 + .6/2 -6) -- (-2 - .6 +2,-5 - .6/2 -6) ;  \draw[color=black] (-2 - .6 +2,-5 - .6/2 -6) node[left] {$F_3'$}; 

\draw (1,-9.5) -- (1.5,-12); \draw (1,-9.5) node[above] {$F_2'$};
\draw (1,-11.4) -- (3.5,-12.6); \draw (2.8,-11.7) node[left] {$F_1'$};
\draw (2.5,-12.4) -- (6 + .1,-12.4); \draw[color=black] (6 + .1,-12.4) node[right] {$Z'$};


\draw (-2 -.6 ,-14 + .6/3) -- (1 + .6,-15 - .6/3) ;  \draw[color=black] (-.5,-14.5) node[above] {$E_1'$};
\draw [fill=black] (-2,-14) circle (1.6pt);
\draw [fill=black] (1,-15) circle (1.6pt);

\draw plot [smooth,tension=1] coordinates {(-4,-15) (-2,-14) (0 -1.5,-16+.4)}; \draw (-4,-15) node[left] {$W'$};

\draw [fill=black] (-8 +2,-4 -6) circle (1.6pt);
\draw [fill=black] (-6 +2,-5 -6) circle (1.6pt);
\draw [fill=black] (-4 +2,-4 -6) circle (1.6pt);
\draw [fill=black] (-2 +2,-5 -6) circle (1.6pt);
\draw [fill=black] (0 +2.,-4 -6) circle (1.6pt);
\draw [fill=black] (2 +2,-5 -6) circle (1.6pt);

\draw (-6 - .6*3/5,-10 + .6) -- (-3,-15); \draw[color=black] (-4.5,-12.5) node[left] {$F'$};
\draw (1 - .5,-15 - .5*3/5) -- (6,-12); \draw[color=black] (3.5,-13.5) node[below] {$E'$};

\end{scriptsize}
\end{tikzpicture}
\caption{Configuration of curves on $\wt S'$}\label{2024_07_23_20:20}
\end{figure}

We finally explain how the fibration $f:\wt S\to \PP^1$ covers the quasi-elliptic fibration $f':\wt S'\to \PP^1$. To this end we consider the commutative diagram
\[
\begin{tikzcd}
    S \ar[d,dashed] \ar[r] & \PP^2 \ar[d,dashed] \\
    S' \ar[r] & \PP^2 
\end{tikzcd}
\]
where the dashed arrows represent the rational maps
\[ ((x:y:z),(t_0:t_1)) \mapsto ((x^2:y^2:xz),(t_0:t_1)) \quad \text{and} \quad (x:y:z)\mapsto (x^2 : y^2 : xz). \]
These maps are undefined only at the points $Q\in S$ and $(0:0:1)\in \PP^2$ respectively.
We analyse how the corresponding rational map
\[ \wt S \ttto \wt S' \]
transforms the fibres over $\PP^1$.
The fibres of $\wt \tau: \wt S \to \PP^1$ over the points different from $(1:0)$ and $(0:1)$ are applied to the corresponding fibres of $\wt \tau' : \wt S' \to \PP^1$ by the quadratic transformation $(x:y:z)\mapsto (x^2 : y^2 : xz)$.
The exceptional curves $E_1$, $E_3$, $F_1$, $F_3$, $F_5$, $F_7$, $F_9$, $F_{11}$ with odd indices are contracted to points,
the exceptional curves $E_2$, $E=E_4$, $F_2$, $F_4$, $F_6$, $F_8$, $F_{10}$, $F=F_{12}$ with even indices are mapped isomorphically onto the curves $E_1'$, $E'=E_2'$, $X'$, $F_3'$, $F_4'$, $F_5'$, $F_6'$, $F'=F_7'$, 
and the curves $X$, $Z$ and $W$ are mapped onto the curves $F_2'$, $Z'$ and $W'$ by inseparable morphisms of degree two.
However, the curve $F_1' \subset \wt S'$ is not covered by any curve on $S$, and so the rational map $\wt S \ttto \wt S'$ is not a morphism.
More precisely, the rational map $\wt S \ttto \wt S'$ is undefined only at the intersection point of $X$ and $Z$. This indeterminacy can be resolved by a unique blowup, whose exceptional line is mapped isomorphically onto the curve $F_1'$.

\end{document}